\documentclass[a4paper, 11pt]{article}

\input xy
\xyoption{all}

\usepackage{amsmath,amsthm,amssymb}
\usepackage[page,header]{appendix}
\usepackage{morefloats}
\usepackage{color}
\usepackage{makeidx}
\usepackage{hyperref}

\makeindex 

\usepackage{latexsym, amsmath, amssymb, amsfonts, amscd}
\usepackage{amsthm}
\usepackage{t1enc}
\usepackage[mathscr]{eucal}
\usepackage{indentfirst}
\usepackage{graphicx, pb-diagram}
\usepackage{enumerate}
\usepackage[all,poly,web,knot]{xy}

\newcommand{\Title}{Title}

\numberwithin{equation}{section}

{\theoremstyle{definition}\newtheorem{definition}{Definition}[section]

\newtheorem{defititle}[definition]{\Title}

\newtheorem{remark}[definition]{Remark}

\newtheorem{ep}[definition]{Example}
\newtheorem{eps}[definition]{Examples}}
\newtheorem{prop}[definition]{Proposition}
\newtheorem{proposition-definition}[definition]{Proposition-Definition}
\newtheorem{lemma}[definition]{Lemma}
\newtheorem{thm}[definition]{Theorem}
\newtheorem{cor}[definition]{Corollary}

\newcommand{\cF}{\mathcal{F}}

\newcommand{\cI}{\mathcal{I}}
\newcommand{\Q}{\mathbb{Q}}
\newcommand{\cU}{\mathcal{U}}

\newcommand{\ie}{{\it i.e.}\/ }
\newcommand{\eg}{{\it e.g.}\/ }
\newcommand{\cf}{{\it cf.}\/ }

\newcommand{\vX}{\mathfrak{X}}

\def\gpd{\,\lower1pt\hbox{$\longrightarrow$}\hskip-.24in\raise2pt
             \hbox{$\longrightarrow$}\,}

\renewcommand{\latticebody}{\drop@{ }}

\newcommand{\N}{\ensuremath{\mathbb N}}
\newcommand{\Z}{\ensuremath{\mathbb Z}}
\newcommand{\R}{\ensuremath{\mathbb R}}

\newcommand{\T}{\ensuremath{\mathbb{T}}}

\newcommand{\g}{\ensuremath{\mathfrak{g}}}

\newcommand{\cN}{\mathcal{N}}

\newcommand{\bt}{\mathbf{t}}                  %target
\newcommand{\bs}{\mathbf{s}}                  %source
\def\act{\mathbin{\hbox{$<\kern-.4em\mapstochar\kern.4em$}}}
\def\ract{\mathbin{\hbox{$\mapstochar\kern-.3em>$}}}

\def\PB(#1,#2,#3,#4){\left\{\begin{matrix}#1&\!\!\!\stackrel{?}{\longrightarrow}&\!\!\!#2\\
\downarrow&&\!\!\!\downarrow\\
#3&\!\!\!\stackrel{?}{\longrightarrow}&\!\!\!#4\end{matrix}\right\}}

\def\pb(#1,#2,#3,#4){ \hom(#1 \to #3, #2 \to #4)}

\newcommand{\p}[1]{\partial #1} %\p{x}
\begin{document}

\begin{center}
{\Large\bf  Smoothness of holonomy covers for singular foliations and essential isotropy
\bigskip

{\sc by Iakovos Androulidakis and Marco Zambon}
}

\end{center}

{\footnotesize
National and Kapodistrian University of Athens
\vskip -4pt Department of Mathematics
\vskip -4pt Panepistimiopolis
\vskip -4pt GR-15784 Athens, Greece
\vskip -4pt e-mail: \texttt{iandroul@math.uoa.gr}
  
\vskip 2pt Universidad Aut\'onoma de Madrid
\vskip-4pt Departamento de Matematicas 
\vskip-4pt and ICMAT(CSIC-UAM-UC3M-UCM)
\vskip-4pt Campus de Cantoblanco, 28049 - Madrid, Spain
\vskip-4pt e-mail: \texttt{marco.zambon@uam.es, marco.zambon@icmat.es}
}
\bigskip
\everymath={\displaystyle}

\date{today}

\begin{abstract}\noindent
Given a singular foliation, we attach an ``essential isotropy'' group to each of its leaves, and show that its discreteness is the integrability obstruction of a natural Lie algebroid over the leaf. We show that a condition ensuring discreteness is the local surjectivity of a transversal exponential map associated with the maximal ideal of vector fields prescribed to be tangent to the foliation.

 The essential isotropy group is also shown to control the smoothness of the holonomy cover of the leaf (the associated fiber of the holonomy groupoid), as well as the {smoothness} of the associated isotropy group. Namely, the (topological) closeness of the essential isotropy group is a necessary and sufficient condition for the holonomy cover to be a smooth (finite-dimensional) manifold and the isotropy group to be a Lie group.

These results are useful towards understanding the normal form of a singular foliation around a compact leaf. At the end of this article we briefly outline work of ours on this normal form, to be presented in a subsequent paper. \end{abstract}

%\begin{abstract}\noindent
%Given a singular foliation, a certain ``essential isotropy'' group is naturally attached to each of its leaves. It is shown that the holonomy cover of a leaf (the associated fiber of the holonomy groupoid) is a smooth manifold and the isotropy group of the leaf is a Lie group if and only if its essential isotropy group is closed. Also, it is proven that the discreteness of this essential isotropy group is the integrability obstruction of a natural Lie algebroid over the leaf. A condition ensuring this discreteness is provided.  

%These results are useful towards understanding the normal form of a singular foliation around a compact leaf. At the end of this article we briefly outline work of ours on this normal form, to be presented in a subsequent paper. \end{abstract}

\setcounter{tocdepth}{2} %hides subsubsections in TOC
\tableofcontents

\section*{Introduction}
\addcontentsline{toc}{section}{Introduction}

\noindent\textbf{Background and motivations. }

An important tool for the study of regular foliations is the notion of holonomy. In particular, the holonomy groupoid of a regular foliation is the starting point for many constructions in foliation theory. In this paper we will consider the much larger class of \emph{singular} foliations, by which we mean a choice  of $C^{\infty}(M)$-module $\cF$ of the vector fields of $M$, which is locally finitely generated and stable by the Lie bracket\footnote{{This definition arises naturally from the work of Stefan \cite{Stefan} and Sussmann \cite{Sussmann}. Note that an old question is whether every singular foliation arises from a Lie algebroid. In Lemma \ref{noal} we give a counterexample in terms of a singular foliation as a module of vector fields.}}. 
%(Note that this definition arises naturally from the work of Stefan \cite{Stefan} and Sussmann \cite{Sussmann}.) 
 Recall that the holonomy of singular foliations was first studied by Dazord in \cite{Dazord:1984}, \cite{Dazord:1985}. 
%But the fact that the relevant holonomy groupoid usually has a very bad topology makes the study of singular foliations in the literature disproportionate to regular ones, despite the fact that the singular case is the most common one. 
%The notion of holonomy groupoid of singular foliations  was not established for a long time; 
For a particular case of ``quasi-regular'' foliations Debord \cite{Debord1} \cite{Debord2} constructed the holonomy groupoid and showed that it is smooth. 

A different construction of the holonomy groupoid for \textit{any} singular foliation was given by the first author and Skandalis in \cite{AndrSk}. It realizes the groupoid  as a quotient of certain (smooth) covers of its open sets (``bi-submersions''). It is a topological groupoid, and its topology is usually quite ``bad'', reflecting the fact that it arises from a singular foliation.  The key feature of this construction, distinguishing it from the previous ones, is that the holonomy covers of a leaf that it induces keep track of the dynamics of the foliation, namely the choice of module $\cF$, rather than the partition of $M$ to leaves (note that in the singular case it is possible for two different modules of vector fields as above to induce the same partition to leaves). The construction of \cite{AndrSk} was motivated by and allows the generalization of Noncommutative Geometry results to the singular case: In \cite{AndrSk1} a longitudinal pseudodifferential calculus was developed, and in \cite{AndrSk2} its associated index theory was given.

In the current article, motivated by questions on the topology of a singular foliation, we prove a few smoothness results about the holonomy groupoid. The main topological invariant we would like to understand eventually is the stability of the foliation around a compact leaf, and its relation with the dynamics of the foliation. Namely to generalize the \emph{local Reeb stability theorem} (see for instance \cite{MM}) to the singular case in a way that keeps track of the choice of module $\cF$.  At the end of this article we outline the relation of the smoothness of the holonomy cover with stability around a compact leaf in the previous sense, as well as work of ours on the latter, which is presented in a separate paper \cite{AZ2}. 

Other topological invariants we would like to eventually understand for singular foliations, and their relation to the smoothness are:
1) Characteristic classes; in the regular case the crucial ingredient for the formulation of characteristic classes is the action of the holonomy groupoid on the normal bundle of the foliation (see \cite{CM}). In the singular case this action cannot even be written down, as neither the holonomy groupoid nor the normal bundle are smooth. The development of characteristic classes will allow calculations for the longitudinal analytic index.
%2) A second example is the stability of the foliation around a compact leaf, namely the local Reeb stability theorem (see \cite{MM}). The smoothness of the holonomy cover of the leaf plays a crucial role in this result. 
2) In Noncommutative Geometry, having the smoothness of the holonomy covers simplifies a great deal the pseudodifferential calculus and index theory developed in \cite{AndrSk1} and \cite{AndrSk2}, as it allows to define the left-regular representations (although, as it was shown in these papers, the existence of left-regular representations is not necessary for the theory to work). 

\vspace{2mm}
\noindent\textbf{Results. }

 The article at hand is the first part of our effort to understand the previous questions on the topology of  singular foliations. 
 %We outline further results in \S \ref{subsection:outline} and will appear in a forthcoming paper \cite{AZ2}, show that what really matters is the smoothness of the holonomy cover of the leaf $L_x$ at $x$, namely the smoothness of the source-fibers $H_x$ of the holonomy groupoid constructed in \cite{AndrSk}. In the current paper we discuss exactly this smoothness.
We determine conditions for the smoothness of the 
 holonomy cover of the  leaf $L_x$ at a point $x$, which is just  the source-fiber $H_x$ of the holonomy groupoid constructed in \cite{AndrSk}, and of the restriction $H_{L_x}$ of the holonomy groupoid to $L_x$.
This issue was discussed briefly already in \cite[Rem. 3.13]{AndrSk}, where it was suggested that the smoothness is related with the integrability obstruction of a certain Lie algebroid attached to the leaf.  More precisely, we show:
\begin{itemize}
\item The smoothness of $H_{L_x}$ and $H_x$ is controlled by a certain ``essential isotropy'' group attached to the leaf $L_x$ . Precisely,  $H_{L_x}$ and $H_x$ are smooth if and only if the essential isotropy group of $L_x$ is closed (Thm. \ref{HL} and Prop. \ref{source}). 
%\mcomment{I removed in the isotropy group $H_x^x$. One direction of this "iff" is probably wrong, isn't it}
This is always satisfied for regular leaves.
\item There is a canonical transitive Lie algebroid $A_{L_x}$ attached to the leaf $L_x$. It integrates to $H_{L_x}$ if and only if the essential isotropy group is discrete (Thm. \ref{ALintegr}).
\end{itemize}
 
 {Recall that in \cite[\S 3.3]{AndrSk} a question was posed regarding the relation between the smoothness of $H_x$ and the Crainic-Fernandes obstruction to the integrability of $A_{L_x}$ \cite{CrF}. This relation should be codified in some relation between our essential isotropy group and the monodromy group of $A_{L_x}$. The full clarification of such a relation is an issue we hope to clarify in a different paper. In Remark \ref{MRui}  we examine the case when the essential isotropy is discrete (then it is automatically central): It turns out that then the monodromy group lies inside our essential isotropy group.}

Although it seems, from studying examples, that the essential isotropy groups may well always be discrete, we could not provide a concrete proof for this. However, we are able to present a condition which implies the discreteness of the essential isotropy groups.
Choose a point $x \in M$ , a transversal $S_x$ to the leaf $L_x$, and consider the    infinite-dimensional Lie algebra  $I_x \cF_{S_x}$, where $\cF_{S_x}$ is the restriction of $\cF$ to the transversal and $I_x$ is the collection of functions on $S_x$ which vanish at $x$. Then (Thm. \ref{thm:secord}):
\begin{itemize}
\item 
Assume that  for any smooth time-dependent vector field  $\{X_t\}_{t\in [0,1]}$ in $I_x\cF_{S_x}$, there exists a vector field $Z\in I_x\cF_{S_x}$ and 
a neighborhood $S'$ of $x$ in $S_x$ such that $exp(Z)|_{S'}=\phi|_{S'}$, where $\phi$ denotes the time-1 flow of $\{X_t\}_{t\in [0,1]}$.
Then the essential isotropy group of $L_x$  is discrete. 
\end{itemize} 
(This condition can interpreted in terms of local surjectivity of the exponential map of $I_x \cF_{S_x}$.)

The Lie algebra $I_x \cF_{S_x}$ turns out to be a crucial ingredient in the understanding of the transversal action of the holonomy groupoid. This action, together with  the smoothness results achieved in this note,  provides the model for a generalization of the local Reeb stability theorem to singular foliations.  This paper ends with an outline of these results, which will appear in \cite{AZ2}.  

\noindent\textbf{Organization of the paper. }In \S \ref{section:foliations}, \S \ref{section:holgpd} we recall the construction of the holonomy groupoid from \cite{AndrSk} and prove certain technical results that we will need later in the sequel. We also discuss  bi-submersions somewhat more deeply, showing that they provide an equivalent definition of a foliation. \S \ref{section:isotropy} introduces the essential isotropy groups and proves that $H_x$ is smooth iff the associated essential isotropy group is closed. In \S \ref{section:integr} we study the transitive Lie algebroid $A_L$ associated to a leaf $L$ and prove that the discreteness of the essential isotropy group attached to $L$ is the integrability obstruction of $A_L$. In the same section we discuss the condition on $I_x \cF_{S_x}$   mentioned above. In \S \ref{section:outlook} we outline the results of the forthcoming paper \cite{AZ2}.

\noindent\textbf{Notation.} Given a manifold $M$, we  use $\vX(M)$ to denote its vector fields, and  $\vX_c(M)$ its vector fields with compact support. For a vector field $X$ and $x\in M$, we use $exp_x(X)\in M$ to denote the time-one flow of $X$ applied to $x$. By $\cF$ we will always denote a singular foliation on $M$ and $L$ a leaf. If $X \in \cF$, we use $[X]$ to denote the class $X \text{ mod }I_x\cF$ (here $x\in M$). The notation $\langle X \rangle$ is used to denote either classes under several other equivalence relations or  the foliation generated by $X$.
{Given a vector bundle $E\to M$, we denote by $C^{\infty}(M;E)$ its space of smooth sections.}\\
The holonomy groupoid of a singular foliation is denote by $H$. Further, $H_x=\bs^{-1}(x)$ is the source fiber  and $H_x^x=\bs^{-1}(x)\cap \bt^{-1}(x)$ the isotropy group at $x$; the same notation applies to bi-submersions  $U,W,...$

\noindent\textbf{Acknowledgements.} We would like to thank Marius Crainic and his group, Rui Fernandes, Camille Laurent, Kirill Mackenzie, Ivan Struchiner and especially Georges Skandalis for illuminating discussions. Further we thank Viktor Ginzburg, Janusz Grabowski, Daniel Peralta, Tudor Ratiu for advice on the condition appearing in Thm. \ref{thm:secord}. We also thank Dionysios Lappas and Noah Kieserman for proofreading an earlier version of the manuscript.

I. Androulidakis was partially supported by a Marie Curie Career Integration Grant (FP7-PEOPLE-2011-CIG, Grant No PCI09-GA-2011-290823), by DFG (Germany) through project DFG-Az:Me 3248/1-1 and FCT (Portugal) through project PTDC/MAT/098770/2008. M. Zambon was partially supported by CMUP and FCT (Portugal) through the programs POCTI, POSI and Ciencia 2007, by projects PTDC/MAT/098770/2008 and PTDC/MAT/099880/2008 (Portugal) and by  projects MICINN RYC-2009-04065 and MTM2009-08166-E (Spain). M. Zambon thanks the University of G\"ottingen for  hospitality.

%\section{The holonomy groupoid}\label{section:holgpd}

%In this section we recall the construction of the holonomy groupoid of a singular foliation from \cite{AndrSk}. To this aim we first review singular foliations in \S \ref{subsection:fol}-\ref{subsection:splthm}, we review bi-submersions and their morphisms in \S \ref{subsection:bisubmersions}-\ref{subsection:morbisubbis}, and then in \S \ref{sec:groatlas} we construct the holonomy groupoid as a quotient of bi-submersions. Along the way we provide several examples, and in \S \ref{subsection:ALintro} we  attach naturally a transitive Lie algebroid to each leaf of a singular foliation.

\section{Foliations and the transitive Lie algebroid over a leaf}\label{section:foliations}

In this section we review singular foliations from \cite{AndrSk} and construct explicitly a transitive Lie algebroid $A_L$ over a leaf $L$. Namely, in \S \ref{subsection:fol} we recall the definition of a singular foliation adding some clarifications and several examples. Then, \S \ref{subsection:splthm} provides an alternative version of \cite[Prop. 1.12]{AndrSk} which will be useful to us in this sequel, and \S \ref{subsection:ALintro} gives the construction of $A_L$.

\subsection{Foliations}\label{subsection:fol}
 
Let $M$ be a smooth manifold. Given a vector bundle $E \longrightarrow M$ , we denote by
$C^{\infty}_{c}(M;E)$   the $C^{\infty}(M)$-module of compactly supported smooth sections of $E$. 
\begin{enumerate}
    \item\label{foliation} A {\em{(singular) foliation}} on $M$ is a locally finitely generated submodule $\cF$ of the $C^{\infty}(M)$-module $\vX_c(M)=C^{\infty}_c(M;TM)$, stable by the Lie bracket. 
     
    Recall from \cite[\S1.1, item 4]{AndrSk} that the restriction of $\cF$ to an open $U \subseteq M$ is the submodule of $C^{\infty}(U;T_U M)$ generated by $f\cdot \xi|_U$, where $f \in C^{\infty}_c(U)$   and $\xi|_U$ is the restriction of a vector field $\xi \in \cF$ to $U$. In this sequel we will denote the restriction of $\cF$ to an open $U$ by $\cF_U$.
   
    Given this, the module $\cF$ being locally finitely generated (\cf \cite[\S1.1, items 3,5]{AndrSk}) means that for every point of $M$ there exists an open neighborhood $U$   and a finite number of vector fields $X_1,\ldots X_{n}
     \in \vX(M)$ such that for every $k = 1,\ldots,n$ we have 
     %$C^{\infty}(U)_cX_k|_{U} \subset \cF_{U}$ and 
 $\cF_{U} = C^{\infty}_c(U)X_1|_{U} + \ldots + C^{\infty}_c(U)X_{n}|_{U}$.

    \item Stefan \cite{Stefan} and Sussmann \cite{Sussmann} showed that such modules induce a partition of $M$ to injectively immersed submanifolds (leaves). Throughout this sequel we'll call a leaf $L$ \textit{regular}   if there exists an open neighborhood $W$ of $L$ in $M$, where the dimension of $L$ is  equal to the dimension of any other leaf intersecting $W$. Otherwise $L$ will be a \textit{singular} leaf.
 
    \item A foliation $\cF$ provides an alternative topology for the manifold $M$ (``longitudinal smooth structure'',  see \cite[\S1.3]{AndrSk}). Denoting $M_\cF$ the manifold $M$ endowed with this topology, we have:
\begin{itemize}
 \item $M_\cF$ is a totally disconnected manifold, with connected components the leaves. %Obviously, connected components may have different dimensions (but still the smooth structure is well defined).
 \item The identity map $Id : M_\cF \to M$ is continuous but not open. 
\end{itemize}

In the sequel, when we refer to the smooth structure (and corresponding topology) on a leaf $L$, we mean the above longitudinal smooth structure. 
For embedded leaves, this topology coincides with relative topology to the usual topology of $M$. When $L$ is immersed but not embedded though (\eg a (regular) dense leaf of the irrational rotation foliation of the torus), the two topologies differ.
% we merely view $L$ as a connected component of $M_{\cF}$. For reasons of economy, in the sequel we do not examine separately the embedded and the immersed case. Instead, we give our proofs in terms of the usual topology (\ie as if the leaf $L$ was embedded), and provide remarks describing the modifications that are required order to employ the longitudinal smooth structure in case $L$ is immersed.

    \item\label{fol:restriction}      
    Let $(M,\cF)$ be a foliation and $W$ an open subset of $M$. We discussed in item \eqref{foliation} the restriction of $\cF$ to $W$. On the other hand, let $\{W_i\}_{i \in I}$ be an open cover of $M$ and $\cF_i$   a foliation on every $W_i$. If $\cF_i$ and $\cF_j$ agree on $W_i\cap W_j$ for every $i,j \in I$, then  using a partition of unity we can construct a unique foliation on $M$ which restricts to $\cF_i$ for every $i \in I$.

    \item\label{fol:pullback}  We can pull a foliation back over a {smooth map}: If $p: N \longrightarrow M$ {is smooth} then $p^{-1}(\cF)$ is the submodule 
of   $\vX_c(N)$ consisting of vector fields $Y$ such that $y\mapsto dp(Y_y)$ lies in $p^*(\cF)$. The latter is defined as the submodule of $C^{\infty}_{c}(N;p^*(TM))$ generated by $f\cdot(X\circ p)$ for $f\in C^{\infty}_{c}(N)$ and $X\in \cF$.  
In order words,   $p^{-1}(\cF)$ consists of  $C^{\infty}_{c}(N)$-linear combinations of vector fields on $N$ which are projectable and project to elements of $\cF$.

%Explicitly:
%        \begin{enumerate}
%            \item Let $p : N \longrightarrow M$ be a smooth map and $\E$ a submodule of $C^{\infty}_{c}(M;E)$. Denote $p^{*}\E$ %the submodule of $C_{c}^{\infty}(N;p^{*}E)$ generated by $f(\xi \circ p)$, where $f \in C^{\infty}_{c}(N)$ and $\xi \in \E$.
%            \item Denote $p^{-1}(\F)$ the submodule of $C^{\infty}_{c}(N;TN)$
%            \[
%            p^{-1}(\F) = \{ X \in C^{\infty}_{c}(N;TN) : dp(X) \in p^{*}(\F) \}
%            \]
%        \end{enumerate}

    \item    For $x \in M$ the space $\cF(x) = \{X\in \cF : X_x = 0\}$ is {Lie subalgebra}  of $\cF$. Let $I_{x} = \{f \in C^{\infty}(M) : f(x)=0\}$. Then $I_x\cF \subset \cF(x)$ is a {Lie} ideal. Hence $\g_x = \cF(x)/I_x\cF$ is a Lie algebra, called the \textit{infinitesimal isotropy} of $\cF$ at $x$.
%\end{enumerate}

\begin{lemma}\label{gxtri}
Let $L$ be a leaf of $\cF$.  The Lie algebras $\g_x$ are isomorphic for all $x\in L$. Further, for any $x\in L$,  $\g_x=\{0\}$ if{f} $L$ is a regular leaf.
\end{lemma} 
\begin{proof} The first statement follows from the fact that the group $\exp(\cF)$, generated by $exp(X)$ with $X\in \cF$,   acts transitively on the leaf $L$. We now prove the second statement. If $L$ is a regular leaf, then nearby the point $x$ we can find generators of $\cF$ which are linearly independent at $x$, implying that $\cF(x)=I_x\cF$. If $L$ is a singular leaf, pick a neighborhood $W$ in $M$ of some $x\in L$,
% such that $\cF|_W$ is not a regular foliation,
and pick a set of generators $X_1,\dots,X_n$ of $\cF$ defined on $W$. We may assume that this is a minimal set of generators, i.e. we may assume that none of the $X_i$ can be written as a $C^{\infty}(W)$  linear combination of the others. Further, as $\{X_i(x)\}$ spans $T_xL$ and $W$ contains leaves of dimension $> dim(L)$,
%$L$ is a singular leaf,
we may assume that $X_1(x)$ is a linear combination of the remaining $X_i(x)$, in other words  $X_1\in \cF(x)$. However $X_1\notin I_x\cF$ because if we could write $X_1=\sum_{i=1}^n  f_iX_i$ with $f_i \in I_x$ then we would have $X_1=\sum_{i\neq 1} {f_i}/{(1-f_1)}X_i$, which contradicts the minimality assumption. \end{proof}

%\begin{enumerate}
    \item[g)]  The {\em{fiber}} of $\cF$ at $x$ is the quotient $\cF_x = \cF/I_x\cF$. The {\em{tangent space of the leaf}} is the image $F_x$ of the evaluation map $ev_x : \cF \to T_xM$. They are both finite-dimensional linear spaces. We have the extension 
\begin{eqnarray}\label{extn:fiber}
0 \to \g_x \to \cF_x \stackrel{ev_x}{\longrightarrow} F_x \to 0. 
\end{eqnarray}
Notice that the Lie bracket on $\g_x$ does not extend to $\cF_x$.  

    \item[h)] The dimension of $F_x$ is lower semi-continuous while the dimension of $\cF_x$ is upper semi-continuous. The union of regular leaves is a dense open subset of $M$ and the two fibres coincide in this set. The {\it co-dimension} of $\cF$ is the upper semi-continuous function $codim : M \to \Z$ which maps every $x \in M$ to the co-dimension of the leaf $L_x$ at $x$.
\end{enumerate}

\begin{eps}\label{eps:1.1}
\begin{enumerate}[(i)]

\item \label{1.1firstitem} Let $X$ be a vector field on $M$ and $\cF:=\langle X \rangle$ the foliation generated by $X$,  that is, $\cF=C^{\infty}_c(M)X$. Denote by $\{X=0\}$ the vanishing set of $X$. Then Lemma \ref{gxtri} shows that $\g_x=\{0\}$ everywhere except if $x\in   \partial\{X=0\}$ (these are exactly the singular leaves), in which case $\g_x\cong \R$   due to the fact that  $dim(\cF_x)\le 1$ since $\cF$ has just one generator. The short exact sequence \eqref{extn:fiber} implies that $\cF_x\cong \R$ except if $x$ lies in the interior of $\{X=0\}$, where $\cF_x$ vanishes. 
     \begin{enumerate}
    \item \label{eps:1.1(iii)} Consider the foliation $\cF$ on $\R^2$ induced from the   action of $S^1$ by rotations. Its regular leaves are circles concentric at the origin, and the origin is a singular leaf. The module $\cF$ is generated by the vector field $X = x\partial_y - y\partial_x$. We have  $\cF_{q} = \R$ at every point $q$,   $\g_0=\R $, whereas $\g_q$ vanishes at every $q \neq 0$.
    
   \item \label{eps:1.1(ii)} Consider the action of $SO(2)\subset SO(3)$ on $S^2$ by rotations about the $z$-axis. The module $\cF$ is generated by the image of the infinitesimal action. We have $\g_N = \g_S = Lie(SO(2))\cong \R$, where $N,S$ are the two poles.

 \item \label{eps:1.1(iv)} The foliation $\cF$ on $\R$ generated by $x\partial_x$ has $\cF_x = \R$ at every point, whereas $\g_x$ vanishes at every point except for
  $\g_0\cong\R $.

\item   \label{eps:1.1(vii)} Consider the foliation $\cF = \langle f\partial_x \rangle$ on $\R$, where $f$ is some smooth function that vanishes exactly on  $\{x\leq 0\}$.   Its leaves are $L^+=\R^+$ and $L_x = \{x\}$ for every $x\leq 0$. All leaves are regular except for $L_0 = \{0\}$.
Then $\cF_x$ vanishes if $x<0$ and is one-dimensional otherwise. The infinitesimal isotropy $\g_x$ vanishes at every $x \neq 0$ except at $0$, where  $\g_0\cong\R $. Notice that $\cF$ vanishes in every order on every $x\leq 0$.    
%half-line. It has $\cF_x = \R$ everywhere, except at $\cF_0 = 0$.

\end{enumerate}

    \item   \label{eps:1.1(i)} In \cite[Prop. 1.4]{AndrSk} it was shown that if $V$ is a (finite-dimensional) vector space and $\cF$ is the foliation defined by a linear action $G \subseteq GL(V)$ then $\cF_0 = \g$, where $\g$ is the Lie algebra of $G$. At points $x\neq 0$ there is a natural map from the isotropy Lie algebra of the action at $x$ to $\g_x$, which is always surjective but might fail to be injective: $\R^2-\{0\}$ is a regular leaf of the action of $GL(\R^2)$ on $\R^2$, hence for any $x\in \R^2-\{0\}$ we have $\g_x=\{0\}$.

    \item \label{eps:1.1(vi)} Let $\cF^k$ the module of vector fields in $\R^2$ vanishing  ``to order $k$'' at $(0,0)$ It is generated by $x^i y^j \partial_x, x^i y^j \partial_y$ for all $i,j{\ge 0}$ such that $i+j=k$. Then $\cF^k_{(0,0)} = \g_{(0,0)} = \R^{2k+2}$.

    \item \label{eps:1.1(viii)} Consider the foliation $\cF$ on $\R$ generated by $f_1\partial_x, f_2\partial_x, \ldots , f_k\partial_x$, where the $f_i$ vanish at $0$  to order at least 1, that is, $f_i(0)={f_i}'(0)=0$. Then $$[f_i\partial_x,f_j\partial_x] = (f_i {f_j}'-f_j {f_i}')\partial_x \in I_x\cF$$ hence the Lie algebra $\g_x$ is abelian.

\item Any Lie algebroid $A \to M$ gives rise to a singular foliation $\cF:=\{\rho(a):a\in C^{\infty}_c(A;M)\}$, where $\rho\colon A \to TM$ is the anchor map.
\begin{lemma}\label{noal}
Not all singular foliations arise from Lie algebroids.
\end{lemma}
\begin{proof}
If $\cF$ is a foliation arising from a Lie algebroid $A$, then  $dim(\g_x)$ is bounded above by the rank of $A$ for all $x\in M$ (here $\g_x$ denotes the infinitesimal isotropy  of $\cF$).
Indeed consider $\ker(\rho_x)$, the isotropy of the Lie algebroid $A$ at $x$. There is a well-defined linear map $\ker(\rho_x) \to \g_x$ mapping $\overline{a}$ to $\langle \rho(a)\rangle$, where $a\in \Gamma(A)$ is any extension of $\overline{a}$. This map is surjective  since any element of $\g_x$ is represented by some vector field $X\in \cF$ vanishing at $x$, so that  $X=\rho(a)$ for some $a\in C^{\infty}_c(A;M)$ with $\rho_x(a_x)=0$.

For any $k\ge 1$, consider the foliation $\cF^{k} $ on $\R^2$ generated by      
    $(x-k)^i y^j \partial_x, (x-k)^i y^j \partial_y$ for all $i,j{\ge 0}$ such that $i+j=k$ (a trivial variation of item (\ref{eps:1.1(vi)})).
We have  $\cF^{k}_{(k,0)} = \R^{2k+2}$. Now take the foliation $\cF$ generated by $\cup_{k\ge 1}\varphi_k\cF^{k}$,
where  $\varphi_k$ is a fixed choice of bump function on $\R^2$ with small support concentrated around the point $(k,0)$. Then the infinitesimal isotropy of $\cF$ satisfies $\g_{(k,0)}=\R^{2k+2}$. In particular, the dimensions of the infinitesimal isotropies of $\cF$ are not bounded above, so $\cF$ can not arise from a Lie algebroid.
\end{proof}
{It is not clear whether the analog of Lemma \ref{noal} holds if $M$ is compact, or even if it holds locally on $M$.}
 \end{enumerate}
\end{eps}

\subsection{The splitting theorem for foliations}\label{subsection:splthm}

We show that locally every foliation is the product of a foliation vanishing at a point and a one-leaf foliation, in  analogy with the splitting theorem for Poisson structures \cite{We83}. {This is done modifying slightly the proof of  \cite[Prop. 1.12]{AndrSk}.}
 
Given singular foliations $(M_1,\cF_1)$ and $(M_2,\cF_2)$, their product is defined as   $(M_1\times M_2,\cF_1\times \cF_2)$, where $\cF_1\times \cF_2$ 
is given by $C_c^{\infty}(M_1\times M_2)(X_1+X_2)$ with $X_1$ and $X_2$ are trivial extensions to $M_1\times M_2$ of vector fields lying in $\cF_1$ and $\cF_2$ respectively.

%is generated  (as a submodule of the compactly supported vector fields on $M_1\times M_2$) by $(X_1)^H$ and $(X_2)^V$ where  $X_1\in \cF_1$ and $X_2\in \cF_2$.

% and $f,g\in C_x^{\infty}(M_1\times M_2)$. Here $(X_1)^H|_{(x_1,x_2)}:=(X_1)|_{x_1}$ and  $(X_2)^V|_{(x_1,x_2)}:=(X_2)|_{x_2}$ for all   $(x_1,x_2)\in M_1\times M_2$. 

%denotes the horizontal lift of $X_1$ by the projection  $M_1\times M_2\to M_1$ using the connection given by

\begin{prop}\label{transversal}
{\em\bf{(Splitting theorem)}} Let $(M,\cF)$ be a manifold with a foliation and $x \in M$, and set $k := dim (F_{x})$. Let $\hat{S} $ be  a slice at $x$, that is, an embedded submanifold such that $T_x\hat{S}\oplus F_x=T_xM$.

 Then there exists an open neighborhood $W$ of $x$ in $M$ and a diffeomorphism of foliated  manifolds 
 \begin{equation}\label{diff}
(W,\cF_W)\cong (I^k,TI^k) \times (S,\cF_S). 
\end{equation}
 Here {$\cF_W$ is the restriction of $\cF$ to $W$,} $I:=(-1,1)$, $S:=\hat{S} \cap W$ and   $\cF_S$ consists of the restriction to $S$ of vector fields in $W$ tangent to $S$ (so $\cF_S=\iota^{-1}\cF$ for $\iota \colon S \hookrightarrow W$). 
 %Further $I:=(0,1)$ is endowed with the regular foliation $TI$. 
 %\mcomment{I don't know how to say it better, since we want vfs with compact support, and $I$ is not compact}
 
In particular, if  we denote by $s_1,\dots,s_k$ the canonical coordinates on $I^k$ and $X_1,\dots,X_l$ are generators of $\cF_S$, then
$\cF_W$ is generated by $\partial_{s_1},\dots, \partial_{s_k}$ and the (trivial extensions of) $X_1,\dots,X_l$.
\end{prop}
\begin{proof}
We proceed by induction. For $k=0$ the result is clear, so assume $k\ge1$ and
take $X\in \cF$ to be a vector field with $X(x)\neq 0$. Let $\hat{V} \subset M$ be a submanifold with $\hat{S} \subset \hat{V} $ and $T_x\hat{V} \oplus \R X(x)=T_xM$. The
  proof of  \cite[Prop. 1.12]{AndrSk}  shows that there is    a neighborhood $W$ of $x$ in $M$ and  a diffeomorphism
  \begin{equation}\label{diffeo1}
  W\cong I\times V
\end{equation}
 mapping the vector field $X$ to ${\partial_s}$, where $s$ is the canonical coordinate on 
$I$ and $V=\hat{V} \cap W$. Under this identification, every compactly supported vector field on $W$ is of the form $fX+Z$, where $Z$ is tangent  $\{s\}\times V$ for all $s$. Further it shows that   $fX+Z\in \cF$ if{f} $Z|_{\{s\}\times V}\in \cF_{\{s\}\times V}$ for all $s$, where the latter consists of the restrictions to $\{s\}\times V$ of  vector fields  in $\cF$ tangent to it.
 
We infer that  $\cF=  TI\times \cF_V$, where $\cF_V:=\cF_{\{0\}\times V}$. That is, eq. \eqref{diffeo1} is  a  diffeomorphism of   manifolds  with foliations 
\begin{equation}\label{diffeo2}
 (W,\cF_W)\cong (I,TI)\times (V,\cF_V)
\end{equation}

 Now $(F_V)_x=F_x\cap T_xV$ has dimension $k-1$ and $S$ is a slice to $(V,\cF_V)$ at $x$, so by the induction hypothesis there is a  diffeomorphism of   manifolds  with foliations
\begin{equation}\label{diff2}
(V,\cF_V)\cong  (I^{k-1},TI^{k-1})\times (S,\cF_S). 
\end{equation}
(Notice that the foliation on $S$ induced by $\cF_V$ agrees with the one induced by $\cF$, which is $\cF_S$.)
Inserting eq. \eqref{diff2} into eq. \eqref{diffeo2}  we obtain the seeked eq. \eqref{diff}.
%$$(W,\cF|_W)\cong (I,TI)\times (I^{k-1},TI^{k-1})\times (S,\cF_S)\cong (I^{k},TI^{k})\times (S,\cF_S).$$
\end{proof}

\begin{eps}
%\label{eps:transv}
\begin{enumerate}[(i)]
\item Let $m$ be a positive integer and consider the foliation $\cF_m = \langle x^{m}\partial_{x}\rangle$ on $\R$. Its leaves are $(-\infty,0), \{0\}$ and $(0,+ \infty)$. If $x$ is non-zero, a transversal at $x$ is just a point with the obvious (zero) foliation. At zero a transversal is $ (\R,\cF_m)$.

\item Consider the action of $S^{1}$ on $\R^{2}$ by rotations. The leaves of $(\R^{2},\cF)$ are concentric circles with a singularity at zero. A transversal at zero is $(S_x,\cF_{S_x}) = (\R^{2},\cF)$. At an $x \neq 0$, any (affine) line $S_x \cong \R$ passing through $x$ is a transversal and $\cF_{V}$ is the foliation by points (generated by the zero vector field in the identification $S_x \cong \R$). In this case there exists a canonical choice of $S_x$, namely the line connecting $x$ with the origin.
\end{enumerate}
\end{eps}

\begin{remark}\label{isogx}
Fix a slice $S$ at $x$ transverse to the foliation $\cF$.  By definition $(F_S)_x = F_x\cap T_x S = \{0\}$, so the exact sequence \eqref{extn:fiber} applied for the foliation $(S,\cF_S)$ gives $(\cF_S)_x = (\g_S)_x$. Further, the latter is canonically isomorphic  to $\g_x$.

Indeed, by the splitting theorem, in every class of $\g_x=\cF(x)/I_x\cF$ there is a representative $Y$ tangent to $S$. Mapping the class of such a   $Y$  to the class of $Y|_S$ gives a well-defined map $\cF(x)/I_x \cF \to \cF_S/I_x \cF_S$. It is surjective, as we can extend an element of $\cF_S$ to one of $\cF(x)$. Further both vector spaces have the same dimension, as a consequence of the splitting theorem
(notice that $\cF_I(x)/I_x \cF_I=\{0\}$ since the foliation $\cF_I$ on $I$ is regular).
 \end{remark}

\subsection{The transitive Lie algebroid $A_L$ over a leaf}\label{subsection:ALintro}

Let $L$ be a leaf of the foliation $(M,\cF)$. There exists a transitive Lie algebroid $A_L$ over $L$, described in \cite[Remark 1.16]{AndrSk}. Its integrability is important because if $A_L$ is integrable then the construction of pseudodifferential calculus in \cite{AndrSk1} is simplified a great deal. This is discussed thoroughly in the beginning of \S \ref{subsection:ALintegr}.

In this subsection we assume that the leaf $L$ is  \emph{embedded} (and not just immersed), and we make explicit the Lie algebroid $A_L$. Consider $\cF/I_L\cF$, where $I_L$ is the space of functions in $C^{\infty}(M)$ which vanish on the leaf $L$.

\begin{prop}\label{A_Lvb} Let $L$ be an embedded leaf.
 \begin{enumerate}
\item 
 $A_L:=\cup_{x\in L} \cF_x$ is endowed with the structure of a smooth vector bundle   over $L$. 
 %of rank $dim(\cF_x)$ (for $x\in L$) 
 \item  $\cF/I_L\cF$ is a $C^{\infty}(L)$-module,
and  $\cF/I_L\cF= C^{\infty}_c(A_L;L)$, where the latter denotes the compactly supported sections. 
%Further, any point $x\in L$ has a neighborhood $U$ such that any section of $A_L|_U$ is the restriction of an element of $\cF/I_L\cF$.
\end{enumerate}
 \end{prop}
\begin{proof} 
a) Cover $L$ by open subsets $\{W_{\alpha}\}_{\alpha \in A}$ of $M$ as in the splitting theorem (Prop. \ref{transversal}), i.e. so that there exists an isomorphism of foliated manifolds
\begin{equation}\label{isoalpha}
\psi_{\alpha} \colon (W_{\alpha},\cF_{W_{\alpha}}) \cong (I^k,TI^k)\times (S_{\alpha},\cF_{S_{\alpha}}).
\end{equation} Here $S_{\alpha}$ is a slice transverse to $L$ at some fixed point $x_{\alpha}$ and $I=(-1,1)$. We denote  $L_{\alpha}:=\psi_{\alpha}^{-1}(  I^k\times \{x_{\alpha}\})$.

Consider  the disjoint union of trivial vector bundles 
\begin{equation}\label{trivvb}
\bigcup_{\alpha \in A} (\cF_{x_{\alpha}} \times L_{\alpha})
\end{equation}
over $\cup_{\alpha \in A}L_{\alpha}$. Notice that  $\psi_{\alpha}$ and the affine structure of $I^k$
  provides identifications $\cF_{x_{\alpha}}\cong \cF_{x}$ for all $x\in L_{\alpha}$.
If $x\in L_{\alpha}\cap L_{\beta}$,    we therefore obtain an isomorphism
$$(\cF_{x_{\alpha}},\{x\}) \cong \cF_x \cong (\cF_{x_{\beta}},\{x\})$$
%as in Remark \ref{isogx}, 
providing an equivalence relation on the vector bundle \eqref{trivvb}.
The quotient  is a vector bundle $A_L \to L$. 
Notice that, as a set, $A_L$ is just the union $\cup_{x\in L} \cF_x$.
%Then $E\oplus TL\to L$ is the seeked vector bundle, since 
%$$\cF/I_L\cF \to C^{\infty}(L;E\oplus TL),\;\; X \text{ mod }I_L\cF \mapsto 
%(X \text{ mod }I_x\cF)+X|_x$$
%is a well defined map which is an isomorphism onto its image.

b) To show that $\cF/I_L\cF$ is a $C^{\infty}(L)$-module, put $f[X] = [\tilde{f}\cdot X]$ for every $f \in C^{\infty}(L)$ and $X \in \cF$. Here $\tilde{f}$ is any smooth extension of $f$ from $L$ to $M$, which exists since $L$ is an embedded submanifold.
Since any two extensions of $f$ differ by an element of $I_L$,
the above definition does not depend  choices, and makes $\cF/I_L\cF$ a $C^{\infty}(L)$-module.

Given $X\in \cF$, there is an embedding $\tau$ that associates to $[X]\in \cF/I_L\cF$ the   element of $C^{\infty}_c(A_L;L)$ given by $$L \to A_L, x \mapsto X \text{ mod }I_x\cF.$$ (This section of $A_L$ has compact support because $X$ is a compactly supported vector field). In order to show that $\tau$ is surjective, we cover $L$ by open subsets $\{W_{\alpha}\}_{\alpha \in A}$ of $M$ as in the splitting theorem and show:

Claim: \emph{For every $\alpha \in A$, the map $\cF_{W_{\alpha}}/I_{L_{\alpha}}\cF_{W_{\alpha}}\to C^{\infty}(L_{\alpha},A_L|_{L_{\alpha}})$
 obtained restring $\tau$ is surjective.}

%consider the neighborhood $U:=L_{\alpha}$ and 
Since both sides are $C^{\infty}(L)$-modules, it is sufficient to show that for any constant section $c$ of the trivial vector bundle $A_L|_{L_{\alpha}}\cong \cF_{x_{\alpha}} \times L_{\alpha}$
there exists $Y\in \cF_{W_{{\alpha}}}$ so that $\tau([Y])=c$. 
%We do this only for $x=x_{\alpha}$ for the sake of simplicity. 
Since $c$ is a constant section, we can view it as an element 
 $c\in \cF_{x_{\alpha}}=T_{0}I^k\oplus (\cF_{S_{\alpha}})_{x_{\alpha}} $ using eq. \eqref{isoalpha}. The first component of $c$ and any lift of the second component of $c$ to 
$\cF_{S_{\alpha}}$ give rise to an element $Y \in \cF_{W_{{\alpha}}}$, invariant in the $I^k$ direction, with the required property. This proves the claim.

The surjectivity of $\tau$ now follows: given $s\in C^{\infty}_c(A_L;L)$,
there is a finite subset $B$ of the index set $A$ such that the support of $s$ is contained in $\cup_{\alpha\in B} L_{\alpha}$. Using an associated partition of unity we can write $s=\sum_{\alpha\in B} s_{\alpha}$ where the  elements $s_{\alpha}$ have support   in $L_{\alpha}$. By the claim we have $s_{\alpha}=\tau[Y_{\alpha}]$ some  $Y_{\alpha}\in \cF$. So we conclude that 
$s =\tau[ \sum_{\alpha\in B} Y_{\alpha}]$. 
\end{proof}

From  Prop.\ref{A_Lvb} it follows that the Lie bracket of $\cF$ and 
 the surjective map $\rho : C^{\infty}(L;A_L) \to \vX(L), \rho[X] = X|_{L}$ make the triple $(A_L,[\cdot,\cdot],\rho)$ into a transitive Lie algebroid.

In fact, the evaluation map of eq. \eqref{extn:fiber} is nothing else than the restriction of $\rho$ to a fiber $(A_L)_x$. Whence the kernel of the anchor map $\rho$ is the field of Lie algebras $\g_L = \cup_{x \in L}\g_x$. This is actually a locally trivial Lie algebra bundle from the general theory of transitive Lie algebroids (\cf \cite{KCHM}).

\begin{remark}\label{rem:immleaf}
The construction of $A_L$ for an immersed leaf $L$ (endowed with the 
longitudinal smooth structure)
is exactly as in part a) of the above Prop. \ref{A_Lvb}. However in that case $\cF/I_L\cF$ -- which is not even a $C^{\infty}(L)$ module -- is not contained in the space of sections of $A_L$. On the other hand, on open sets $L_{\alpha}$ of $L$, we have $C^{\infty}(L_{\alpha},A_L|_{L_{\alpha}})=\cF_{W_{\alpha}}/I_{L_{\alpha}}\cF_{W_{\alpha}}$.   Notice that using a partition of unity for $L$ we can glue elements $\xi_\alpha \in \cF_{W_{\alpha}}/I_{L_{\alpha}}\cF_{W_{\alpha}}$ to obtain elements of $C^{\infty}(L;A_L)$.
\end{remark}

\section{The holonomy groupoid}\label{section:holgpd}

In \S \ref{subsection:bisubmersions} we review the notion of a bi-submersion introduced in \cite{AndrSk} and clarify its role in the study of foliations in \S \ref{subsubsection:absbis}. Then, \S \ref{subsection:morbisubbis} gives explicit formulas of morphisms of bi-submersions that will be used in \S \ref{section:isotropy} to study the smoothness of the holonomy covers. For the convenience of the reader, we provide in \S \ref{sec:groatlas} a brief review of the construction of the holonomy groupoid given in \cite{AndrSk} as a quotient of an atlas of bi-submersions, together with explicit calculations of this groupoid for some foliations.

\subsection{Bi-submersions}\label{subsection:bisubmersions}

Let $(M,\cF)$ be a (singular) foliation. Here we recall the notion of bi-submersion from \cite{AndrSk} and discuss how it provides an equivalent definition of a foliation.
\begin{enumerate}
    \item A {\em{bi-submersion}} of $(M,\cF)$ is a smooth manifold $U$ endowed with two submersions $\bt, \bs : U \longrightarrow M$ satisfying:
        \begin{enumerate}
            \item[(i)] $\bs^{-1}(\cF) = \bt^{-1}(\cF)$,
            \item[(ii)] $\bs^{-1}(\cF) =  C^{\infty}_{c}(U;\ker d\bs) + C^{\infty}_{c}(U;\ker d\bt)$.
        \end{enumerate}
        Recall that $\bs^{-1}(\cF)$  was defined in item \ref{fol:pullback}) of \S \ref{subsection:fol}.
        
We say $(U,\bt,\bs)$ is {\em{minimal}} at $u$ if $dim(U) = dim(M) + dim(\cF_{\bs(u)})$.
    \item\label{phbisubm} Let $x\in M$, and $X_1,\ldots,X_n \in \cF$ inducing a basis of $\cF_x$.  In \cite[Prop. 2.10 a)]{AndrSk} it was shown that there is an open neighborhood $U$ of $(x,0)$ in $M \times \R^n$ such that $(U,\bt_U,\bs_U)$ is a bi-submersion minimal at $(x,0)$, where $\bs_U(y,\xi)=y$ and
   $\bt_U(y,\xi) = exp_y(\sum_{i = 1}^n \xi_i X_i)$. (Recall that the latter is the image of $y$ under the time-$1$ flow of $\sum_{i = 1}^n \xi_i X_i$.) Bi-submersions arising this way are called  {\em{path holonomy bi-submersions}}.  

    \item Let $(U_i,\bt_i,\bs_i)$ be bi-submersions, $i = 1,2$. Then $(U_i,\bs_i,\bt_i)$ are bi-submersions, as well as $(U_1\circ U_2,\bt,\bs)$ where $U_1\circ U_2 = U_1 \times_{\bs_1,\bt_2}U_2$, $\bt(u_1,u_2) = \bt(u_1)$ and $\bs(u_1,u_2) = \bs(u_2)$. They are called the \emph{inverse} and \emph{composite} bi-submersions respectively.
\end{enumerate}
   
\begin{eps}\label{eps:1.2}
\begin{enumerate}
   \item[(i)] In example \ref{eps:1.1}(i)(a) the
path holonomy  bi-submersion, minimal at $x=0$, associated to $X\in \cF$ is  given by $U \subset \R^2 \times \R$, $\bs : U \to \R^2$ the projection and $\bt : U \to \R^2$  the map 
%$$((x,y),\varepsilon) \mapsto 
$$\left(\left(\begin{array}{cc}  x  \\y \end{array}\right),\varepsilon\right)
 \mapsto \left(\begin{array}{cc} cos(\varepsilon) & -sin(\varepsilon)\\
 sin(\varepsilon) & cos(\varepsilon) 
   \end{array}\right)
\left(\begin{array}{cc}  x  \\y \end{array}\right)
$$
   \item[(ii)] In example \ref{eps:1.1}(i)(c) the
path holonomy  bi-submersion, minimal at $x=0$, associated to $x\partial_x$ is $(U,\bt,\bs)$ where $U \subset \R \times \R$, $\bs$ is the first  projection and $\bt(x,\varepsilon) = xe^{\varepsilon}$.
\end{enumerate}
\end{eps}

\subsubsection{Abstract bi-submersions}\label{subsubsection:absbis}

Note that bi-submersions determine the foliation $\cF$ quite directly. Indeed, if $f : U \to M$ is a {bi-}submersion, the foliation $f^{-1}(\cF)$ on $U$ obviously determines the foliation $\cF$ on $f(U)$. This way a bi-submersion $(U,\bt,\bs)$ determines $\cF$ on $\bs(U)$ and $\bt(U)$. 

In fact, this observation leads to an abstraction of the notion of bi-submersion. Although we will not use this abstraction in the sequel, we present it here, as it allows the notion of bi-submersion to be better understood. In particular, it exhibits that bi-submersions are actually an equivalent definition of a foliation. This material was communicated to us by G. Skandalis. Let us start with a simple lemma:

\begin{lemma}\label{lem:submfol}
Let $f : N \to M$ be a surjective submersion with connected fibres. Let $\cF_N$ be a foliation on $N$ such that  $C^{\infty}(N;\ker df) \subseteq \cF_N$. Then there exists a unique foliation $\cF$ on $M$ such that $f^{-1}(\cF) = \cF_N$.
\end{lemma}

\begin{proof} 
Assume first that $N = M \times \R^k$ and $f$ is the first projection. In this case $\cF_N$ is the direct sum of its vertical part generated by $\partial/\partial x_i$ (where $(x_1,\ldots,x_k)$ are coordinates of $\R^k$) and a horizontal part $\cF_{hor}$. Now $\cF_N$ is invariant by translations in the $\R^k$ direction (\cf \eg \cite[Prop 1.6]{AndrSk}), whence its horizontal part is constant along $\R^k$. It is the set of smooth functions with compact support from $\R^k$ to a foliation $\cF_{M}$ of $M$. Now let $\iota : M \to M \times \R^k, x \mapsto (x,0)$. It is transverse to $\cF_N$. Since $\cF_N$ is of the form $f^{-1}(\cF_M)$ we have $\iota^{-1}(\cF_N) = \cF_M$ and uniqueness follows.

In general, $N$ is the union of open sets $W_i {\cong} U_i \times \R^k$, where the $U_i$ are open sets in $M$ {and $f|_{W_i}$ corresponds to the projection $U_i \times \R^k \to U_i$}. The restriction of $\cF_N$ to $W_i$ is of the form $f^{-1}(\cF_i)$, where $\cF_i$ is a foliation on $U_i$. By uniqueness, $\cF_i$ and $\cF_j$ agree on $f(W_i\cap W_j)$. By connectedness of the fibres, for every $x \in M$ and every $i,j \in I$ such that $x \in U_i\cap U_j$, there exists a finite sequence $i = i_1,\ldots,i_r = j$ such that $x \in f(W_{i_{\ell}} \cap W_{i_{\ell +1}})$ for every $1\leq \ell \leq r$. It follows that $\cF_i$ and $\cF_j$ coincide in a neighborhood of $x$. Using partitions of unity we deduce that $\cF_i$ and $\cF_j$ coincide in $U_i\cap U_j$, whence there exists a foliation $\cF_M$ whose restriction to $U_i$ is $\cF_i$. The foliation $\cF_M$ is unique in each $U_i$ by the first case, therefore it is unique.
\end{proof}

\begin{definition}\label{dfn:abstrbisubm}
\begin{enumerate}
\item Let $M,N$ be manifolds. An \emph{abstract bi-submersion} between $M$ and $N$ is a triple $(U,\bt,\bs)$, where $U$ is a manifold and $\bs : U \to M$, $\bt : U \to N$ are submersions with connected fibres such that $C^{\infty}(U;\ker d\bs) + C^{\infty}(U;\ker d\bt)$ are foliations, \ie stable under Lie brackets.
\item Let $M, N_1, N_2$ be manifolds. Let $(U_1,\bt_1,\bs_1)$ be a bi-submersion between $M$ and $N_1$, and $(U_2,\bt_2,\bs_2)$ a bi-submersion between $M$ and $N_2$. For every $i = 1,2$ denote $\cF_i$ the foliation $C^{\infty}(U_i;\ker d\bs_i) + C^{\infty}(U_i;\ker d\bt_i)$ of $U_i$. Put $U = U_1 \times_M U_2$  and let $p_i : U \to U_i$ be the projection. The bi-submersions $(U_1,\bt_1,\bs_1)$ and $(U_2,\bt_2,\bs_2)$ are said to be \emph{compatible} if $p_1^{-1}(\cF_1) = p_2^{-1}(\cF_2)$.
\end{enumerate}
\end{definition}

By Lemma \ref{lem:submfol}, an abstract bi-submersion defines a foliation on $\bs(U)$ and on $\bt(U)$. Also, by uniqueness in Lemma \ref{lem:submfol}, two bi-submersions $(U_1,\bt_1,\bs_1)$ and $(U_2,\bt_2,\bs_2)$ are compatible if and only if the foliations they define agree on   $\bs_1(U_1)\cap \bs_2(U_2)$. The following proposition follows thanks to item \eqref{fol:restriction} in \S \ref{subsection:fol}.

\begin{prop}\label{prop:bisubmfol}
Let $M$ and $\{N_i\}_{i\in I}$ be manifolds. For every $i \in I$ let $(U_i,\bt_i,
\bs_i)$ be an abstract bi-submersion between $M$ and $N_i$. Suppose that:
\begin{enumerate}
\item $\bigcup_{i \in I}\bs_i(U_i) = M$, and
\item for every $i,j \in I$ the bi-submersions $(U_i,\bt_i,\bs_i)$ and $(U_j,\bt_j,\bs_j)$ are compatible.
\end{enumerate}
Then there exists a unique   unique foliation $\cF$ on $M$ such that for all $i \in I$, $$\bs_{i}^{-1}(\cF) = C^{\infty}(U_i;\ker d\bs_i) + C^{\infty}(U_i;\ker d\bt_i).$$
\end{prop}

Proposition \ref{prop:bisubmfol} shows that a foliation can be defined from abstract bi-submersions. Notice that our definition of an abstract bi-submersion (Def. \ref{dfn:abstrbisubm}) actually uses the notion of foliation. It is possible to abstract further the definition of a bi-submersion and completely do away with the notion of foliation in the definition, but this is far beyond the scopes of this sequel.

\subsection{Morphisms of bi-submersions and bisections}\label{subsection:morbisubbis}

Let us now recall from \cite{AndrSk} the notion of a morphism of bi-submersions, as well as the notion of a bisection. We will be using both of them throughout this paper. 

\begin{enumerate}
    \item Let $(U,\bt_{U},\bs_{U})$ and $(V,\bt_{V},\bs_{V})$ be two bi-submersions. A {\em{morphism of bi-submersions}} is a smooth map $f : U \longrightarrow V$ such that $\bs_{V} \circ f = \bs_{U}$ and $\bt_{V} \circ f = \bt_{U}$.
    \item A {\em{bisection}} of $(U,\bt,\bs)$ is a locally closed submanifold $V$ of $U$ on which the restrictions of $\bs$ and $\bt$ are diffeomorphisms to open subsets of $M$. We say that $V$ is an {\em{identity bisection}} if $\bs|_V=\bt|_V$. A bisection is necessarily the image of a smooth map $\psi : M_0 \to U$ such that $\bs \circ \psi=Id_{M_0}$, where $M_0$ is an open subset of $M$. In case $(U,\bt,\bs)$ is a path holonomy bi-submersion $U\subseteq M \times \R^n$, a bisection is necessarily the graph of a smooth map $\phi : M_0 \to \R^n$. We will often switch freely between a bisection and the corresponding maps.     
    \item We say that $u \in U$ {\em{carries}} the foliation-preserving local diffeomorphism $\psi$ if there is a bisection $V$ such that $u \in V$ and $\psi = \bt\mid_V \circ (\bs\mid_V)^{-1}$. For instance, if $U\subseteq M \times \R^n$ is a path holonomy bi-submersion which is minimal at $(x,0)$ and $\phi : V \to \R^n$ a smooth map which vanishes at $x$ then $(x,0)$ carries the local diffeomorphism $y \mapsto exp_y (\sum_{i=1}^n \phi_i(y)X_i)$. ({Notice that in general this local  diffeomorphism is not the time-1 flow of a vector field in any obvious way, due to the dependence of $\phi_i$ on $y$.)}
        Putting $\phi$ the constant zero map, we see that $(x,0)$ carries the identity as well.  
    \item\label{itd} It was shown in  \cite[Cor. 2.11(b)]{AndrSk} that if $\{(U_i,\bt_i,\bs_i)\}_{i\in I}$ are bi-submersions, $i = 1,2$ then $u_1 \in U_1$ and $u_2\in U_2$  carry the same local diffeomorphism iff there exists a morphism of bi-submersions $g$ defined in an open neighborhood of $u_1 \in U_1$ such that $g(u_1) = u_2$. Such a morphism maps every bisection $V$ of $U_1$ at $u_1$ to a bisection $g(V)$ of $U_2$ at $u_2$.
    %    \item If $f : (U_1,t_1,s_1) \to (U_2,t_2,s_2)$ is a morphism and $U_2$ is minimal at $f(u_1)$ then $df_{u_1}$ is onto. Therefore, there is a neighbourhood $U'$ of $u_1$ in $U_1$ such that the restriction of $f$ to $U'$ is a submersion.
    %    \item For every bi-submersion $(U,t,s)$ and every $u \in U$, there exists a bi-submersion $(U',t',s')$ and $u' \in U'$ such that $U'$ is minimal at $u'$ and carries at $u'$ the same diffeomorphisms as $U$ at $u$. Hence there is a neighbourhood $W$ of $u$ in $U$ and a submersion which is a morphism $f : (W,t\mid_W,s\mid_W) \to (U',t',s')$.
\end{enumerate}

\begin{ep}\label{ep:1.3}
 %\item[(i)] The map $\kappa_U$ discussed in the proof of \ref{localgp} is a morphism.
For the circle action on $\R^2$ we saw that a path holonomy bi-submersion is given by $U \subseteq \R^2 \times \R$ as in example \ref{eps:1.2}(i). The action groupoid $S^1 \ltimes \R^2$ is also a  bi-submersion.  The familiar exponential map $\R \to S^1$ provides a (locally defined) map $U \to S^1 \ltimes \R^2$ by $(\vec{\lambda},\theta) \mapsto (e^{i\theta},\vec{\lambda})$. This is a morphism of bi-submersions.
\end{ep}

The following Lemma is a special case of \cite[Prop. 2.10 b)]{AndrSk}.
% which will be used in the sequel. 
We include a proof, along the lines of that of \cite[Prop. 2.10 b)]{AndrSk} but more explicit. Recall 
that the latter states that if $U,U'$ are bi-submersions with $x\in \sharp U$ and $U'$ is a   path holonomy bi-submersion  minimal at $x$, then -- shrinking $U$ if necessary --  there exists a morphism of bi-submersions $\Phi \colon U \to U'$ such that $(x,0)$ lies in its image, and   furthermore $\Phi$ is a submersion.

\begin{lemma}\label{lemma:unique2}
Fix $x \in M$ and consider vector fields $X_1,\cdots,X_n$ and $X'_1,\cdots,X'_n$ in $\cF$ whose images $\{[X_1],\cdots,[X_n]\}$ and $\{[X'_1],\cdots,[X'_n]\}$ form bases of $\cF_x$. Let $(U,\bt,\bs)$ and $(U',\bt',\bs')$ be the corresponding path holonomy bi-submersions. There exists a morphism $\Phi : U \to U'$ defined in a neighborhood of $(x,0)$ and fixing $(x,0)$, which is a diffeomorphism onto its image.
\end{lemma}
\begin{proof}
Let $\{Y_i\}$ be vertical lifts of the $\{X_i\}$ w.r.t. the target  map $\bt$ of the bi-submersion $U$, and $\{{Y'}_i\}$ vertical lifts of
of the $\{{X'}_i\}$ w.r.t. the target  map $ {\bt}'$ of the bi-submersion $U'$. Since $\{X_i'\}$ is a generating set for $\cF$ in a neighborhood of $x$  (which we will still denote by $M$ abusing notation) \cite[Prop 1.5(a)]{AndrSk},
there exists smooth functions $\{c_{il}\}_{i,l\le n}$ satisfying $X_i=\sum_l c_{il}X_l'$ for all $i$. In particular, 
$\sum_i k_i [X_i]=\sum_{i,l} k_i c_{il}(x)[X'_l]\in \cF_x$. We define $\Phi$ by 
\begin{eqnarray}\label{Phi}
\exp_{(y,0)}(\sum_i \lambda_iY_i)&\mapsto& \exp_{(y,0)}\Big(\sum_{i,l}  \lambda_i \bt'^*(c_{il})Y'_l\Big).
\end{eqnarray}
Since $Y_i$ and $Y'_i$ are vertical vector fields, they lie in the kernels of the respective source maps, which are just the projections onto $M$. Hence we have $\bs'\circ \Phi=\bs$. On the other hand, notice that $\sum_l \bt'^*(c_{il})Y'_l$ maps under $\bt'$ to $\sum_l c_{il}X'_l=X_i$, so $exp_{(y,0)}(\sum_{i,l}  \lambda_i \bt'^*(c_{il})Y'_l)$ maps under $\bt'$ to $\exp_{y}(\sum_i \lambda_iX_i)$, which is the image under $\bt$ of $\exp_{(y,0)}(\sum_i \lambda_iY_i)$. This shows that $\bt'\circ \Phi=\bt$ and hence that $\Phi$ is a morphism of bi-submersions. 

To show that it is a diffeomorphism consider its derivative $(\Phi_*)_{(x,0)}$. It maps $Y_i|_{(x,0)}$ to $\sum_{l} c_{il}(x)Y'_l|_{(x,0)}$. Since  $ \{[X_i]\}_{i\le n}=\{\sum_{l} c_{il}(x)[X'_l]\}_{i\le n}$ and $\{[X'_i]\}_{i\le n}$ are both bases of $\cF_x$, the matrix $[c_{il}(x)]_{il}$ is invertible. 

We claim that the $\{Y'_i|_{(x,0)}\}_{i\le n}$ are linearly independent (this immediately implies   that
 $\Phi$ is a local diffeomorphism near $x$ and concludes the proof). 
 Indeed, assume that the constant linear combination $Z:=\sum_{i\le n} a_iY'_i$ vanishes at $(x,0)\in U'$. Choose coordinates along the fibers of $\bs \colon U' \to M$, giving rise to vertical coordinate vector fields $\{\partial_i\}_{i\le n}$. Then $Z=\sum_{i\le n} b_i\partial_i$ for unique $b_i\in C^{\infty}(U')$, which satisfy $b_i(x,0)=0$. We have $\bt_*(Z|_M)=\sum_{i\le n} {b_i}|_M \bt_*({\partial_i}|_M) \in I_x\cF$. (Here we use $\bt_*({\partial_i}|_M) \in \cF$, a direct consequence of the fact that $U'$ is a bi-submersion.) So we have $[\bt_*(Z|_M)]=0\in \cF_x=\cF/I_x\cF$.
At the same time we have  $\bt_*(Z|_M)=\sum_{i\le n} a_iX'_i$, which implies $[\bt_*(Z|_M)]=\sum_{i\le n} a_i [X'_i]$. The linear independence of $\{[X'_i]\}_{i\le n}$ implies that $a_i=0$ for all $i$, proving the claim.
\end{proof}

\subsection{The holonomy groupoid}\label{sec:groatlas}

We end this section with a recollection of the construction of the holonomy groupoid given in \cite{AndrSk} and give a few examples.

Let $\cU = \big((U_{i},\bt_{i},\bs_{i})\big)_{i \in I}$ be a family of bi-submersions. Recall from \cite{AndrSk} that a bi-submersion $(U,\bt,\bs)$ is \emph{adapted} to $\cU$ if for all $u \in U$ there exists an open subset $U' \subset U$ containing $u$, an $i \in I$, and a morphism of bi-submersions $U' \to U_{i}$. We say that $\cU$ is an \textit{atlas} if
\begin{enumerate}
\item $\bigcup_{i \in I}\bs_i(U_{i}) = M$.
\item The inverse of every element in $\cU$ is adapted to $\cU$.
\item The composition $U \circ V$ of any two elements in $\cU$ is adapted to $\cU$.
\end{enumerate}

An atlas $\cU' = \{(U'_{j},\bt_{j},\bs_{j})\}_{j \in J}$ is adapted to $\cU$ if every element of $\cU'$ is adapted to $\cU$. We say $\cU$ and $\cU'$ are \textit{equivalent} if they are adapted to each other.   Given a foliation $(M,\cF)$  \textit{the path holonomy atlas} is the one 
generated by all the path holonomy bi-submersions. (The latter were defined in  \S \ref{subsection:bisubmersions} \ref{phbisubm}.) In other words,  {the path holonomy atlas} consists of finite compositions of path holonomy bi-submersions and their inverses.
 
The groupoid of an atlas $\cU = \big((U_{i},\bt_{i},\bs_{i})\big)_{i \in I}$ is the quotient 
$$G(\cU) :=\coprod_{i\in I}U_i/\sim$$ by the equivalence relation   for which $u\in U_i$ is equivalent to $u'\in U_j$ if there is a morphism of bi-submersions $f:W\to U_j$ defined in a neighborhood $W\subset U_i$ of $u$ such that $f(u)=u'$. 
We denote the canonical quotient map by
$\sharp : \coprod_{i\in I}U_i  {\to} G(\cU).$

 The groupoid of the path holonomy atlas associated to $\cF$ is the \textit{holonomy groupoid} of the foliation $\cF$. It depends on the choice of module $\cF$ rather than the partition to leaves (see Ex. \ref{eps:linact} below). We denote the holonomy groupoid of $\cF$ by $H(\cF)$ or by $H$ (when the foliation $\cF$ is understood).

\begin{eps}\label{eps:linact}
\begin{enumerate}
\item[(i)]  Let $(M,\cF)$ be a foliation and $S$ a transversal to a leaf. Consider the foliation $(S,\cF_S)$ introduced in   \S \ref{transversal}. Its holonomy groupoid is the restriction of the holonomy groupoid of $(M,\cF)$ on $S$, namely $H(S,\cF_S) = H(M,\cF)_S^S$. Notice that if $L$ is a singular leaf then the isotropy groups $H(M,\cF)^x_x$ may not be discrete. Due to Remark \ref{isogx}, this constitutes a difference from the regular case, where the restriction of the holonomy groupoid to a transversal is always \'{e}tale\footnote{A very common technique in the study of regular foliations is to exploit the fact that the holonomy groupoid is always Morita equivalent to an \'{e}tale one (its restriction to a transversal). For instance, in \cite{CM} it plays the crucial role in the extraction of explicit formulas for characteristic classes of regular foliations. Our observation in this example shows that this device cannot be used for singular foliations.}.
\item[(ii)]\label{ex:onevf} Let $X$ be a complete vector field on $M$ and $\cF=\langle X \rangle$. Assume that, for all $x\in \partial \{X=0\}$, every neighborhood of $x$ contains at least one point   the integral curve through which is not periodic, where $\{X=0\}$ denotes the vanishing set of $X$. Then  
$$H(X) =H(X)|_{\{X\neq0\}}\cup \text{Int}\{X=0\} \cup  (\R\times \partial \{X=0\}),$$
where $H(X)|_{\{X\neq0\}}$ is the (smooth) holonomy groupoid of the regular 1-dimensional foliation $\cF = \langle X \rangle$ on $\{X\neq0\}$, and the groupoid structure  on $\R\times \partial \{X=0\}$ is given by addition.

Indeed, the flow of $X$ induces an action of the abelian group $\R$ on $M$. The transformation groupoid $G:=\R \ltimes M \rightrightarrows M$ is a bi-submersion and an atlas for $(M,\cF)$, hence  $H$ is a quotient for the transformation groupoid (see \cite[Ex. 3.4(4)]{AndrSk}). If $x\in \text{Int}\{X=0\}$ is an interior point of $\{X=0\}$, then any two any points $(\lambda,x)$ and $(\eta,x)$ of $G$ are related by a morphism of bi-submersions defined locally, namely the fibre translation by $\eta-\lambda$. If $x$ is a boundary point of $\{X=0\}$, then no morphism of bi-submersions relating distinct points $(\lambda,x)$ and $(\eta,x)$ can exist. If it existed, it would relate distinct points of $\R \times \{z\}$ for all $z$ lying in some neighborhood $U$ of $x$. This would contradict our assumption that there exist a  $y\in U$ such that the target map of $G$  is injective on $\R \times \{y\}$.

\item[(iii)] Consider the foliation $\cF$ on $\R^2$ induced from the   action of $S^1$ by rotations, i.e. $\cF=\langle x\partial_y - y\partial_x\rangle$. (Notice that this vector field does not satisfies the assumption made in ii) above.)
Then $H$ coincides with the transformation groupoid $S^1 \times \R^2 \rightrightarrows \R^2$. Indeed the latter is a bi-submersion and an atlas, so $H$ is a quotient of it. We know that the quotient map restricted to $S^1 \times \{0\}$ is injective  from \cite[Prop. 1.4]{AndrSk}, and further on $\R^2/\{0\}$ the foliation is regular with trivial holonomy, so that $H|_{\R^{2}/\{0\}}=S^1 \times \R^2/\{0\}$.

\end{enumerate}
\end{eps}
 
%\begin{remark}\label{remark:holgpdsmooth}
%In \cite[Rem. 3.13]{AndrSk} a necessary condition for the holonomy groupoid $H$ to be longitudinally smooth (to have smooth $\bs$-fibers) was given. Here, in \S \ref{subsection:isotrsmooth}, \ref{subsection:longitsmooth}, \ref{subsection:ALintegr} we give an explicit obstruction to this, and discuss how the vanishing of this obstruction is related with the integrability of the Lie algebroid $A_L$ we discussed in \S \ref{subsection:ALintro}.
%\end{remark}

\subsection{\texorpdfstring{Smoothness}{Smoothness}}\label{subsubsection:Htop}

 We put here a few clarifications about the notion of smoothness for the holonomy groupoid, following \cite[Remark 3.13, Def. 3.14]{AndrSk}. 
 
%Let $\cF$ be a singular foliation. 
%and $\{(U_i,\bt_i,\bs_i)\}_{i \in I}$ its atlas of path holonomy bi-submersions. 
First, let us recall the topology of the holonomy groupoid:
 $H$ is endowed with the quotient topology, namely the finest topology which makes each quotient map $\sharp \colon U\to H$ continuous for each bi-submersion   in the path holonomy atlas. Note that with this topology $\sharp U$ is open in $H$: The pre-image $\sharp^{-1}(\sharp U)$ consists of points $y$ lying in  bi-submersions $U_{\alpha}$ of the path holonomy atlas such that there exists a morphism of bi-submersions $V_{\alpha} \to U$ for some open neighborhood $V_{\alpha}$ of $x$ in $U_{\alpha}$.  Hence $\sharp^{-1}(\sharp U)$ is open in $ \cup_{\alpha} U_{\alpha}$, and by the definition of quotient topology it follows that $\sharp U$ is open in $H$. %\mcomment{We defined "path holonomy bi-submersion" to mean one of the "basic" ones constructed from a basis of $\cF_x$. I made sure that we use this term consistently with this meaning (as opposed to "bi-submersions in the path holonomy atlas")}

%The following lemma shows that we don't need to consider extra bi-submersions adapted to the path holonomy atlas.
%\begin{lemma}\label{lemma:HLtop}
%If $(V,t,s)$ is a bi-submersion adapted to the path holonomy atlas $\{(U_i,t_i,s_i)\}_{i \in I}$ then $\sharp(V)$ is open in $H$.
%\end{lemma}
%\begin{proof}
%Let $U = \coprod_{i \in I}U_i$. It is a bi-submersion of $\cF$. From \cite[Lemma 4.3]{AndrSk} we have that there exists a bi-submersion $(W,t,s)$ and submersions $p_U : W \to U$, $p_V : W \to V$. Without loss of generality we may assume that $p_V : W \to V$ is surjective (otherwise we restrict to an open subset of $V$). We then have $\sharp(V) = \sharp(p_V(W))=\sharp(p_U(W))$. Since $p_U$ is a submersion $p_U(W)$ is open in $U$, whence $\sharp(V)$ is open in $H$ endowed with the quotient topology.
%\end{proof}

In this sequel we will be concerned with the restriction $H_L = \bs^{-1}(L) = \bt^{-1}(L)$ of $H$ to a leaf $L$, as well as the $\bs$-fiber $H_x = \bs^{-1}(x)$ and the isotropy group $H_x^x = \bs^{-1}(x) \cap \bt^{-1}(x)$, where $x\in L$. 
%We clarify for later use the notion of smoothness for $H_L$, following \cite[Dfn. 3.14]{AndrSk}. 

\begin{definition}\label{dfn:HLtop}
Let $\cF$ be a singular foliation and $L$ a leaf.
%and $\{(U_i,\bt_i,\bs_i)\}_{i \in I}$ its atlas of path holonomy bi-submersions. For every $i \in I$ let $(U_i)_L = \bs_i^{-1}(L) = \bt_i^{-1}(L)$.
 We say that $H_L$ is \emph{smooth} if there exists a differentiable  structure on it such that
% the topology it induces is the finest topology which makes 
 $\sharp \colon U_L \to H_L$ a submersion for all  bi-submersions $(U,\bt,\bs)$ in the path holonomy atlas, where $U_L:=\bs^{-1}(L)$. (If such a differentiable structure    exists, it is unique.)
\end{definition}

%Obviously, this definition means that when $H_L$ is smooth then a typical manifold chart of $H_L$ is $\sharp(U')$, where $U'$ is an open subset of a restricted path holonomy bi-submersion $(U_i)_L$. %An easy adaptation of lemma \ref{lemma:HLtop} shows that we don't need to consider all the bi-submersions adapted to the path holonomy atlas. 

Similarly, when we talk about the smooth structure of $H_x$, we refer to the  differential structure (unique if it exists) such that 
$\sharp \colon U_x:=\bs^{-1}(x) \to H_x$ a submersion for all  bi-submersions $U$ in the path holonomy atlas. The same applies to   $H_x^x$.

\begin{lemma}\label{lem:smooth}
If $H_L$ is smooth (in the sense of Def. \ref{dfn:HLtop}), then it is a Lie groupoid.
\end{lemma}
\begin{proof}
First notice that the topology underlying the differentiable structure
of $H_L$ is exactly the quotient topology discussed above (this is a consequence of the fact that submersions are open maps).

If $U_1$ and $U_2$ are bi-submersions in the path holonomy atlas, then the diagram
\begin{equation*}
\xymatrix{
(U_1 \circ U_2)_L \ar[d]^{\sharp \times \sharp}\ar[dr]^{\sharp } 
% \ar[r]^{mult.} & \Omega_\alpha \ar[d]^{\sharp} \\
&\\
(H_L) _{\bs}\times_{\bt} (H_L) \ar[r]^{\;\;\;\;\;\;\;\;\;mult.}& {H_L} }
\end{equation*}
commutes by the definition of the multiplication on $H_L$. Since $\sharp \times \sharp$ is a submersion, through every point of $(U_1 \circ U_2)_L$ we can find a submanifold $S$ such that $(\sharp \times \sharp)|_S$ is a diffeomorphism onto its image. Hence  the bottom map $mult$ is smooth, as  locally it can be written as $\sharp \circ ((\sharp \times \sharp)|_S)^{-1}$. Similarly one shows that the inversion, the embedding of the identity section, the source map and the target map are smooth.

To show that the target map of $H_L$ is a submersion, take a bi-submersion $U$ in the path holonomy atlas and consider
submanifolds $T\subset U_L $ transverse to the fibers of the submersion $\sharp \colon U_L \to H_L$ (so $\sharp|_T$ is a diffeomorphism onto its image). Since the $\sharp$-fibers are contained in the fibers of $\bt \colon U_L \to L$ (by definition of the equivalence relation $\sim$) and $\bt$ is a submersion, it follows that $\bt|_T \colon T \to L$ is a submersion too. Hence the target map of $H_L$ is a submersion. For the source map one proceeds similarly.
 \end{proof}

\begin{remark} $H_L$ is smooth  for all leaves $L$ if{f} the path holonomy  atlas of $(M,\cF)$, together with  $H$, is a \emph{holonomy pair} in the sense of \cite[Def. 3.14]{AndrSk}. This follows using
Lemma \ref{lem:smooth}.
\end{remark}

To ensure that $H_L$ is smooth, it is sufficient to consider bi-submersions that cover a neighborhood of the identity section $L$: 
 
\begin{prop}\label{prop:allsmooth}
Assume that  for every $x \in L$ there exists a path-holonomy bi-submersion $U$,
%$(U,\bt,\bs)$
minimal at $x$, such that $\sharp \colon U_L \to H_L$ is a submersion. Then $H_L$ is smooth (in the sense of Def. \ref{dfn:HLtop}). 
\end{prop}
\begin{proof} %\mc{DO you agree on the first sentence?}
First notice that  a path holonomy bi-submersion minimal at $x$, say  $(U\subset M\times \R^n, \bt,\bs)$, is isomorphic to its inverse $U^{-1}$, by the map $$U\to U^{-1}, (y,\xi)\mapsto (\bt(y,\xi),-\xi).$$ Hence in the following we will consider only products of   path holonomy bi-submersions (and not of their inverses). It suffices to prove the following claim:

 Claim: \emph{Let $U^1,\dots,U^m$ minimal path-holonomy  bi-submersions.  Then there exists a differential structure on the image of  \begin{equation*}
 %\label{eq:um}
\sharp \colon U^1_L \circ \dots \circ U^m_L\to H_L
\end{equation*}
 such that the above map $\sharp$ is a submersion.}

We prove the claim by induction. The case $m=1$ holds by assumption together with Lemma \ref{lemma:unique2}.  Assume the claim holds for $m-1$, and use the short-hand notation $U^{(m-1)}_L:=U^1_L \circ \dots \circ U^{m-1}_L$. 
By the induction assumption, there exists a smooth structure on 
 $\sharp  U^{(m-1)}_L$ such that $\sharp \colon U^{(m-1)}_L \to H_L$ is a submersion, and similarly for $\sharp  U^{m}_L$.

Consider the following restriction of the multiplication on $H_L$:
\begin{equation*}
mult\colon \sharp U^{(m-1)}_L \;_{\bs}\times_{\bt} \sharp U^m_L \to H_L.
\end{equation*}
For any $q$ in the image of $mult$ we have $$mult^{-1}(q)=\{(q\cdot r^{-1},r):r\in \bs^{-1}(\bs(q))\cap \sharp U^m_L\} .$$
Notice that $\bs^{-1}(\bs(q))\cap \sharp U^m_L$ is an embedded submanifold of  $\sharp U^m_L$, as $\bs \colon \sharp U^m_L \to L$ is a submersion by the proof of Lemma \ref{lem:smooth}.  
%Choose a  smooth $\hat{\sigma} \colon L \to U^m_L$ which is right-inverse to the source map, then  $\sigma := \sharp \circ \hat{\sigma}$ is
Let $\sigma  $ be  a smooth local section of $\bs$.
 Then $$\{(q\cdot \sigma(\bs(q))^{-1},\sigma(\bs(q))): q\in \text{image}(mult)\}$$
is submanifold transverse to the fibers of $mult$, and induces a differentiable structure on $\text{image}(mult)\subset H_L$ such that $mult$ is a submersion. The commuting diagram
\begin{equation*}
\xymatrix{
U_L^{(m-1)} \circ U_L^m \ar[d]^{\sharp \times \sharp}\ar[dr]^{\sharp } 
% \ar[r]^{mult.} & \Omega_\alpha \ar[d]^{\sharp} \\
&\\
\sharp U_L^{(m-1)}  \;_{\bs}\times_{\bt} \sharp U_L^{m}  \ar[r]^{\;\;\;\;\;\;\;\;\;\;\;mult.}& {H_L} }
\end{equation*}
and the induction assumption imply that the diagonal map $\sharp$ is a 
submersion, proving the claim.

 The differentiable structure on an open subset of $H_L$ defined in the above claim is independent of the choice of bi-submersion in the  path holonomy atlas. Indeed if $U,U'$ are two such bi-submersions and $u\in U, u'\in U'$ map to the same point of $H_L$ under the quotient map, then there exists a (smooth) morphism a bi-submersions $\Phi$ defined near $u'$ such that $\sharp_{U'}=\sharp_U \circ \Phi$, showing that $\sharp_{U'}$ is a smooth map for the differentiable structure defined by $U$. 
 %In conclusion, we showed that $H_L$ is smooth in the sense of Def. \ref{dfn:HLtop}.
%The topological groupoid $H_L\rightrightarrows L$ is source-connected, since path holonomy bisubmersions and their inverses are. Hence any symmetric neighborhood $N$ of the set of identities $L$ in $H_L$ generates the whole of $H_L$ \cite[Prop. 1.5.8]{KCHM}. We may assume that $N$ is covered by charts of the form $\sharp(U_{L_W})$. Write $p$ as product of $p_i\in N$, choose smooth bisections $b_i$ through $p_i$, and multiplying them obtain a bisection $b$ through $p$. 
%Right groupoid multiplication by the bisection $b$ provides an  identification 
%$$\text{(neighborhood of $y:=\bt(p)$ in $H_L$)$\cong$ (neighborhood of $p$ in $H_L$)},$$ 
%which we can pre-compose with the homeomorphism $\sharp|_{T^y} \colon T^y \to 
%\text{(neighborhood of $y $ in $H_L$)}$ constructed as above from a path holonomy  bi-submersion  minimal at $y$. Thus we obtain a chart for $H_L$ near $p$.
\end{proof}

\section{\texorpdfstring{Smoothness of the holonomy groupoid $H$ restricted to a leaf}{Smoothness of the holonomy groupoid $H$ restricted to a leaf}}\label{section:isotropy}

 In this section we discuss conditions for the smoothness of the isotropy groups of the holonomy groupoid, of the fibres of its source map, {and of its restriction    to a leaf}. %We will need these results in \S \ref{section:deform}, in order to generalize Heitsch's results on deformations of foliations to the singular case. 
Namely, after some preparation in  \S \ref{subsection:locallgp}-\ref{subsec:intoHxx}, we construct a  morphism   $\varepsilon \colon G_x \to H_x^x$ from the simply connected Lie group integrating the isotropy Lie algebra $\g_x$ into the isotropy group at $x$  of the holonomy groupoid. %This morphism  will be used in \S \ref{section:geomhol} too. It is constructed choosing a bi-submersion but is independent of this choice.
We use this morphism in \S \ref{subsection:isotrsmooth} and \S \ref{subsection:longitsmooth} to show that when $ker(\varepsilon)$ is closed then the isotropy group $H_x^x$ and the restriction  $H_L$ of the holonomy groupoid to a leaf are both smooth. 
%In \S \ref{subsection:ALintegr} we see further that, when $ker(\varepsilon)$ is discrete, the Lie groupoid $H_L$ integrates the Lie algebroid $A_L$.

We are going to consider the \textit{path holonomy atlas}  %\mc{I deleted: an atlas of \textit{path holonomy} bi-submersions}, see \S \ref{subsection:bisubmersions}. 

\subsection{A local group structure attached to bi-submersions}\label{subsection:locallgp}

Let $(M,\cF)$ be a singular foliation. This section proves the following result: If $x \in M$ and $(U,\bt,\bs)$ is a path holonomy bi-submersion, minimal at $(x,0)$, then $U_x^x = \bs^{-1}(x) \cap \bt^{-1}(x)$ is a local Lie group.

\begin{lemma}\label{unique}
Let $(U,\bt,\bs)$ be a path holonomy bi-submersion, minimal at $(x,0)$. 
Let $X \in \cF$ and $Y,Y' \in C^{\infty}(U;\ker d\bs)$ such that $\bt_*(Y) = \bt_*(Y') = X$. Then the restrictions of $Y, Y'$ to $U_x^x$ are equal.
\end{lemma}
\begin{proof}
We have to show that if $Z\in  C^{\infty}(U;\ker d\bs) \cap C^{\infty}(U;\ker d\bt)$
then $Z|_{U_x^x}=0$. Choose $X_1,\dots,X_n \in \cF$ whose images in $\cF_x$ form a basis, and choose lifts $Y_i \in C^{\infty}(U;\ker d\bs)$, that is, $\bt_*Y_i=X_i$. Since the $Y_i$ are linearly independent (\cf {proof of Lemma \ref{lemma:unique2} or} \cite[Prop. 1]{Debord1})   and $\dim(\cF_x)=dim(\ker d\bs)$ there exist (unique) $g_i\in C^{\infty}(U)$ for which $\sum_{i=1}^n g_i Y_i=Z$. Let $u\in U_x^x$, let $D$ be a small submanifold of $U$ through $u$ transverse to the $\bt$-fibers, and define $\tilde{g}_i\in C^{\infty}(M)$ by $(g_i)|_D=\tilde{g}_i \circ \bt|_D$. We have $$0=\bt_*(Z)=\sum_i \tilde{g}_i X_i,$$  
so $$[0]=[\sum_i \tilde{g}_i X_i]=\sum_i \tilde{g}_i(x) [X_i]\in \cF_x.$$
As the $[X_i]$ form a basis of $\cF_x$, we obtain $\tilde{g}_i(x)=0$ for all $i$, so $g_i(u)=0$ for all $i$, showing that $Z_u=0$.   As $u\in U_x^x$ was arbitrary, we conclude that $Z|_{U_x^x}=0$.
\end{proof}
%_{i=1}^n
\begin{lemma}\label{lifti}
Let $(U,\bt,\bs)$ be a path holonomy bi-submersion, minimal at $(x,0)$. There exists a well-defined Lie 
algebra isomorphism  
%\begin{align*}
$$
\delta \colon \g_x \to \{Y\in \vX(U_x^x): Y \text{ admits a $\bt$-projectable extension }\hat{Y}\in C^{\infty}(U,\ker d\bs)\},
\quad [X] \mapsto \hat{Y}|_{U_x^x}  
$$
%\end{align*}
where $X\in \cF(x)$ and $\hat{Y}$ is any lift of $X$ by $\bt$ lying in $C^{\infty}(U;ker d\bs)$.
\end{lemma}
\begin{remark}\label{expl} Let $\{X_i\}_{i\le n} \in \cF$  be vector fields whose images in $\cF_x$ form a basis of $\cF_x$  and such that the images of $\{X_i\}_{i\le \ell}$ for a basis of $\g_x$. Let $U$ be the corresponding path holonomy bi-submersion.
   Choose lifts $Y_i \in C^{\infty}(U;\ker d\bs)$ of the $X_i$. Then the map $\delta$ is given by $[X_i]\mapsto Y_i|_{U_x^x}$ for all $i\le \ell$.
\end{remark}

\begin{proof} We show that $\delta$ is well-defined.
Fix an element of $\g_x$ and choose a representative $X$; then  $\hat{Y}|_{U_x^x}$
is uniquely defined by lemma \ref{unique}. The independence from the choice of representative $X$ is as follows: a lift of $fZ$, where $f\in \cI_x$ and $Z\in \cF$, is given by $\bt^*(f)\cdot\widehat{Z}$, where $\widehat{Z}$ is a lift of $Z$, hence on $U_x^x$ it vanishes.  

The injectivity of the map is clear from Rem. \ref{expl} together with the fact that the $Y_i$ are linearly independent at every point of $U_x^x$.
The surjectivity holds because $Y=\delta[
\bt_* \hat{Y}]$ for any $\bt$-projectable extension $\hat{Y}$ of $Y$. Hence $\delta$ is a linear isomorphism.

To see that $\delta$ is a Lie algebra morphism, take $X,X'\in \cF$ and lifts $\hat{Y}, \hat{Y}'\in C^{\infty}(U;ker(d\bs))$ and use that $[\hat{Y},\hat{Y}']$ is a lift of $[X,X']$, so that $\delta[[X,X']]=[\hat{Y},\hat{Y}']|_{U_x^x}=[Y,Y']=[\delta X, \delta X']$.
\end{proof}

\begin{prop}\label{localgp} Let $(U,\bt,\bs)$ be a path holonomy bi-submersion, minimal at $x$. Let $G_x$ be the simply connected Lie group  integrating $\g_x$ and denote by $\tilde{G_x}$ a neighborhood of the unit.
Shrinking $\tilde{G}_x$ and $U$ if necessary, there is a canonical diffeomorphism
$ \Delta \colon \tilde{G}_x \to U_x^x$ making
$U_x^x$ into a local Lie group with identity element $(x,0)$.
\end{prop}
\begin{proof}
First notice that $U_x^x$ is a submanifold of $U^x_x \subset \{x\}\times \R^n$ of dimension $dim(\g_x)$ containing $(x,0)$.
%an open neighbourhood of the origin in $\R^{dim\g_x}$. 
This is immediate if the vector fields $\{X_i\}\subset \cF$ used to define the bi-submersion  arise from a local splitting of $(M,\cF)$ as $(I^k,TI^k)\times (S,\cF_S)$ as in Proposition \ref{transversal}. The general case holds because any path holonomy bi-submersion minimal at $x$ is isomorphic to one arising from a local splitting of  $(M,\cF)$, by Lemma \ref{lemma:unique2}.

By Lemma \ref{lifti}, $\delta$   is an infinitesimal action of the Lie algebra $\g_x$ 
on the manifold $U_x^x$.  Evaluating every $\delta([X])$ at $x$ we obtain a map $\g_x \to T_x(U_x^x)$ which is injective, as a consequence of the fact that 
 the ${Y_i}|_x$ are linearly independent (by the proof of Lemma \ref{lemma:unique2}). By dimension count, it follows that the previous map is an isomorphism, whence the infinitesimal action $\delta$ is locally free at $x$. 
%It is locally free at $x$, meaning that the evaluation $\g_x \to T_x(U^x_x)$ is injective (see the proof of \cite[Prop. 2.10 b)]{AndrSk}. \mcomment{I don't quite understand that proof, and there should be an easier argument too.} By dimension reason \mcomment{Is this really true?} it is actually an isomorphism.
By standard arguments, $\delta$ can be integrated to a local right group action of   $\tilde{G_x}$ on $U_x^x$.   Evaluation at $x$ gives a map
\begin{equation}\label{GxU}
\Delta \colon \tilde{G}_x \to U_x^x, \quad g \mapsto x\cdot g,
\end{equation}
which (shrinking $\tilde{G}_x$ and $U$ if necessary) is a diffeomorphism. 
\end{proof}

It's easy to see that for any $v\in \g_x$,   under the identification $\Delta$, the infinitesimal generator  $\delta(v)\in \vX(U_x^x)$ corresponds to $\overleftarrow{v}\in \vX(\tilde{G}_x)$, the left-invariant vector field  whose value at the unit is $v$.  Since $exp_{\g}(tv)\subset G$ is the integral curve of  $\overleftarrow{v}$ starting at the identity element, we obtain an alternative formula for $\Delta$ (here $X_i$ and $Y_i$ are as in Remark \ref{expl}):
% as in Prop. \ref{localgp}:
% Choosing vector fields $X_1,\ldots,X_k \in \cF(x)$ whose classes in $\cF_x$ form a base of $\g_x$, and $s$-vertical lifts $Y_1,\ldots,Y_k$ respectively, the explicit formula for the map $\tilde{E}$ is 
\begin{eqnarray}\label{ElocformU}
\Delta \colon \tilde{G}_x \to U_x^x\;,\;
exp_{\g}(\sum_{i=1}^l k_i[X_i]) \mapsto   exp_{(x,0)}(\sum_{i=1}^l k_i Y_i). 
\end{eqnarray}
 
\subsection{\texorpdfstring{The canonical morphism into the isotropy groups of $H$}{The canonical morphism Gx to Hxx}}\label{subsec:intoHxx}

Here we introduce a canonical morphism of local topological groups from $U_x^x$ to the isotropy group $H_x^x$ of the holonomy groupoid.

\begin{lemma}\label{lem:inde} Denote by $\sharp$
 the quotient map onto the holonomy groupoid as in \S \ref{sec:groatlas}.
%\mc{I deleted: $\colon \cup_{\alpha} U_{\alpha} \to H$ the quotient map, where the union is taken over all  path holonomy bi-submersions.}  
The composition
\begin{equation}\label{tvare}
\tilde{\varepsilon} \colon \tilde{G}_x \stackrel{\Delta}{\cong} U_x^x \stackrel{\sharp}{\to} H_x^x
\end{equation}
 is independent of the choice of path holonomy bi-submersion $U$.
\end{lemma}
\begin{proof}
Let $U$ and $U'$ be bi-submersions, minimal at $x$. Then, from Lemma \ref{lemma:unique2} there exists an \emph{isomorphism} of bisubmersions $\Upsilon \colon U \to U'$ defined near $x$ and fixing the point $x$. Since the maps $\delta$ of Lemma \ref{lifti} are given in terms of the source and target maps, $\Upsilon$ intertwines the maps $\delta$ and $\delta'$ of Lemma \ref{lifti}, and therefore intertwines also the diffeomorphisms $\Phi \colon \tilde{G_x} \to U_x^x$ and $\Phi' \colon \tilde{G_x} \to {U'}_x^x$. We conclude by noticing that $\Upsilon$, as every morphism of bi-submersions, intertwines the quotient maps to $H$.
\end{proof} 
 
We prove two properties of the map  $\tilde{\varepsilon}$ of eq. \eqref{tvare}. 
\begin{lemma}\label{openim}
The image of the map $\tilde{\varepsilon}$ is open in $H_x^x$. \end{lemma}
 
 \begin{proof} Recall that $H$ is endowed with the quotient topology given by the map 
$\sharp : \cup_{\alpha} U_{\alpha} \to H$.
One has $\sharp U_x^x=\sharp U\cap H_x^x$, so it suffices to show that $\sharp U$ is open in $H$, which is true by the beginning of  \S \ref{subsubsection:Htop}. 
%The pre-image $\sharp^{-1}(\sharp U)$ consists of points $y$ lying in  path holonomy bi-submersions $U_{\alpha}$ such that there exists a morphism of bi-submersions $V_{\alpha} \to U$ for some open neighborhood $V_{\alpha}$ of $x$ in $U_{\alpha}$.  Hence $\sharp^{-1}(\sharp U)$ is open in $ \cup_{\alpha} U_{\alpha}$, and by the definition of quotient topology it follows that  $\sharp U$ is open in $H$.
\end{proof}
 
 \begin{prop}\label{expismorph} 
 The map $\tilde{\varepsilon}$ in eq. \eqref{tvare}  is a morphism of local topological groups.
\end{prop}
\begin{proof} The map $\tilde{\varepsilon}$ is continuous since $\Delta$ and $\sharp$ are. 
%\mc{I added the summation ranges, they avoid confusion.   We consistently use $n=dim(\cF_x)$ and $\ell=dim(\g_x)$, so $n=k+\ell$}
We check explicitly that formula \eqref{tvare} respects products, using eq. \eqref{ElocformU}. Let  $ \tilde{\g}_x$ be a neighborhood of the origin in $\g_x$ so that the exponential  map $\exp_{\g} \colon \tilde{\g}_x \to \tilde{G}_x$ is a diffeomorphism. The (partial) product on $\tilde{G}_x$ corresponds to the Baker-Campbell-Hausdorff formula (BCH) in $\tilde{\g}_x$.

Let $\{X_i\}_{i\le n} \in \cF$  be vector fields whose images in $\cF_x$ form a basis of $\cF_x$  and such that the images of $\{X_i\}_{i\le \ell}$ for a basis of $\g_x$.  Let $v^{\alpha}=\sum_{i\le \ell} k_i^{\alpha} [X_i ]\in \tilde{\g}_x$ (${\alpha}=1,2$)   so that 
$$v:=BCH(v^1,v^2)=v^1+v^2+\frac{1}{2}[v^1,v^2]+\frac{1}{12}[[v^1,[v^1,v^2]]+\dots$$
 also lies in $\tilde{\g}_x$.  Write  $v = \sum_{i\le \ell} k_i [X_i ]$ for some coefficients  $\overrightarrow{k} \in \R^ \ell $.  %Let $(U,\bt,\bs)$, where $U \subset M \times \R^n$, the minimal bi-submersion at $x$ associated to a choice of local generators $X_i$ of $\cF$.
 % \icomment{The following vector fields are not correct, \eg $(u_1,u_2)$ is not in $U\circ U$. The correct definition is in the sketchy proof I put below.} 
 Consider the   bi-submersion $U\subset M\times \R^n$ associated to $X_1,\cdots,X_n$, and let $\{Y_i\}$ be a choice of vertical vector fields near $(x,0)$ in $U\subset M\times \R^n$ which lifts
 $X_i$ w.r.t. the target map.
 Put $u_{\alpha} :=\Delta(exp_{\g} v^{\alpha})= exp_{(x,0)}(\sum_{i\le \ell} k_i^{\alpha} Y_i) $ and $u = exp_{(x,0)}(\sum_{i\le \ell} k_i Y_i) $. 
Notice that we have $(u_1,u_2)\in U\circ U $, that is $\bs(u_1)=x=\bt(u_2)$, since the vector field $\sum_{i\le \ell} k^2_i Y_i$ maps under $\bt_*$ to $\sum_{i\le \ell} k^2_i  X_i$, which vanishes at $x$.
Since \begin{align*} 
\tilde{\varepsilon}(exp_{\g}v_1 \cdot exp_{\g}v_1)=\tilde{\varepsilon}(exp_\g v)&=\sharp u\\
\tilde{\varepsilon}(exp_{\g}v_1 )\cdot \tilde{\varepsilon}(exp_{\g}v_2 )&=\sharp(u_1,u_2),  
\end{align*}
we have to show that there exists a morphism of bi-submersions $U \circ U \to U$ defined in a neighborhood of $(u_1,u_2)$ such that $(u_1,u_2) \mapsto u$.

  Consider the map
$$\varphi \colon \cF_x \rightarrow \cF,\;\;\; [X_i]\mapsto X_i.$$ 
%a section of the projection $\cF \to \cF_x$ with kernel $I_x \cF$.
The map $\varphi$ is not a Lie algebra homomorphism; rather, $\varphi[v^1,v^2]$ and $[\varphi v^1, \varphi v^2]$ differ by an element of $I_x\cF$. Using the fact that $\varphi v^1$ and
 $\varphi v^2$  are vector fields vanishing on $x$ one shows that 
 $BCH(\varphi v^1,\varphi v^2)$ and $\varphi(BCH(v^1,v^2))=\sum_{i\le \ell} k_iX_i$ differ by an element $Z\in I_x\cF$,   which hence is of the form $\sum_{i\le n} f_i X_i$ where $f_i\in I_x$.
Therefore  
\begin{equation}\label{phi}
BCH(\varphi v^1,\varphi v^2)=\varphi (BCH(v^1,v^2))+Z=\sum_{i\le n} \phi_iX_i
\end{equation}
where $\phi_i=k_i+f_i\in C^{\infty}(M)$ (setting $k_i=0$ for $\ell <i\le n$). Notice that in particular $\phi_i(x) = k_i$.

Consider the following local diffeomorphism of $M$:
\begin{equation}\label{rl}
exp(\sum_{i\le \ell} k_i^1 X_i) \circ exp(\sum_{i\le \ell} k_i^2 X_i) 
%=exp(\varphi v^1)\circ exp(\varphi v^2)
=exp(BCH(\varphi v^1,\varphi v^2))=exp(\sum_{i\le n} \phi_iX_i),
\end{equation}
where we used eq. \eqref{phi} to write down the last term. The local diffeomorphism on the l.h.s. of \eqref{rl} is carried in $U \circ U\subset M\times \R^n \times \R^n$ by the bisection
\begin{equation*}
\{(\bt(y,\overrightarrow{\lambda_y}),exp_{(\bt(y,\overrightarrow{\lambda_y}),0)}(\sum_{i\le \ell} k_i^1 Y_i) )\;\;,\;\;(y,\overrightarrow{\lambda_y}): y\in M\}
\end{equation*}
where we use the short-hand notation $\overrightarrow{\lambda_y}:=exp_{(y,0)}(\sum_{i\le \ell} k_i^2 Y_i) $ for all $y\in M$. This bisection  passes through the point $(u_1,u_2)$ since $\bt(u_2)=x$. The local diffeomorphism on the r.h.s. of \eqref{rl} is carried by the following bisection of $U$:
\begin{equation*}
\{ (y, exp_{(y,0)}(\bt^*\phi_i \cdot Y_i): y\in M\}.
\end{equation*}
This bisection passes through the point $u = exp_{(x,0)}(\sum_{i\le \ell} k_i Y_i)$.
Indeed the curve $$\gamma \colon \epsilon \mapsto exp_{(x,0)}(\epsilon \sum_{i\le \ell}\bt^*\phi_i \cdot  Y_i)$$ maps under $\bt$ to the constant curve at $x$ (since the vector field $\sum_{i\le \ell} \phi_i X_i$ vanishes at $x$), so along $\gamma$ we have $\bt^*\phi_i=\phi_i(x)=k_i$.

In conclusion, we found bisections of $U\circ U$ and $U$, passing respectively through $(u_1,u_2)$ and $u$, which carry the same diffeomorphism.  It follows from item \ref{itd}) in \S \ref{subsection:morbisubbis} 
%\cite[Cor. 2.11]{AndrSk}  
that there exists a morphism of bi-submersions $U \circ U \to U$ such that $(u_1,u_2) \mapsto u$.  
\end{proof}

Last, let us discuss the kernel of $\tilde{\varepsilon}$.

\begin{lemma}\label{lem:ker}
%Consider the map of Prop. \ref{welldef}
%$$E\colon \tilde{\g}_x \overset{\epsilon}{\rightarrow} \R^n \times M \rightarrow H_x^x \quad v \mapsto  (\overrightarrow{k},x) %\mapsto (\overrightarrow{k},x)/\sim.$$ 
%\begin{enumerate}
Let $(U,\bt,\bs)$ a path holonomy bi-submersion minimal at $x\in M$ and consider the map $\tilde{\varepsilon} \colon \tilde{G}_x \to H_x^x$. Then  $g\in ker\tilde{\varepsilon}$ if{f} there exists a bisection of $(U,\bt,\bs)$ through $\Delta(g)$ carrying the identity, where $U$ and $\Delta$ are as in Prop. \ref{localgp}.
%\item For all $l\in \Z$, if $\xi\in ker(\tilde{E})$ and $l\xi\in \tilde{\g}_x$, then $l\xi\in ker(\tilde{E})$  . 
\end{lemma}

\begin{proof}
Let $g \in \ker\tilde{\varepsilon}$, i.e. $\Delta(g)\in \ker\sharp$. Since $(x,0)\in U\subset M \times \R^n$ also belongs to $\ker\sharp$, there exists a morphism of bi-submersions $U\to U$ such that $(x,0) \mapsto \Delta(g) $, and the image of $U\cap (M \times \{0\})$ is a bisection through $\Delta(g)$ carrying the identity. The converse follows by reversing the previous argument, which is allowed due to item \ref{itd}) in \S \ref{subsection:morbisubbis}.
\end{proof}

\subsection{\texorpdfstring{Smoothness of the isotropy groups of $H$}{Smoothness of the isotropy groups of $H$}}\label{subsection:isotrsmooth}

In this section we show that, under suitable hypotheses the isotropy groups of the holonomy groupoid are smooth. 

\begin{prop}\label{thm:integr}
Let $x \in (M,\cF)$   and let $G_x$ be the connected and simply connected Lie group integrating the Lie algebra $\g_x$. There exists a morphism of topological groups  $$\varepsilon \colon G_x \to H_x^x$$ extending the map $\tilde{\varepsilon}$ in eq. \eqref{tvare}.  Further, $\varepsilon$ maps surjectively onto $\mathring{H}_x^x$, the  connected component of the identity in $H_x^x$.
\end{prop}
\begin{proof} 
The map $\tilde{\varepsilon}$ is a morphism of local topological groups by Prop. \ref{expismorph} and   $G_x$ is simply connected, so $\tilde{\varepsilon}$  can be extended uniquely to a topological group morphism $\varepsilon \colon   G_x \to H_x^x$ (\cf \cite[III, Lemma 1, p. 304]{Bou}).

By Lemma \ref{openim}   $\tilde{\varepsilon}(\tilde{G}_x)$ is open in the topological group $H_x^x$, hence any point $h$ in    the  connected component of the identity of $H_x^x$ can be written as a product of elements of $\tilde{\varepsilon}(\tilde{G}_x)\cap (\tilde{\varepsilon}(\tilde{G}_x))^{-1}$. The point $h$ is then the image under $\varepsilon$ of a product of elements of $ \tilde{G}_x$, so $\varepsilon$ is surjective
 {onto $\mathring{H}_x^x$}.

As a Lie group, $G_x$ is second countable, so to prove the continuity of $\varepsilon$ it suffices \cite[\S VI.3]{Ja84} to show that if $(g_n)_{n \in \N}$ is a sequence in $G_x$ which converges to $g \in G_x$ then $(\varepsilon(g_n))_{n \in N}$ converges to $\varepsilon(g)$. The sequence $(g_n g^{-1})_{n\in\N}$ converges to the identity and we may assume that all of its terms except finitely many lie in a neighborhood of the identity in the domain of $\tilde{\varepsilon}$. We have: $$\varepsilon(g_n) = \varepsilon(g_n g^{-1} g) = \tilde{\varepsilon}(g_n g^{-1})\varepsilon(g)$$ We conclude passing to the limit and using the continuity of $\tilde{\varepsilon}$.
  \end{proof}

From Prop. \ref{thm:integr} we have an exact sequence of topological groups \begin{equation}\label{seqgr}
1 \to \ker\varepsilon \to G_x \stackrel{\varepsilon}{\to} \mathring{H}_x^x \to 1. 
\end{equation} The group $\ker\varepsilon$ plays an important role: In \S  \ref{subsection:ALintegr} we will see that the discreteness of the group $\ker\varepsilon$ is the integrability obstruction for the Lie algebroid $A_L$. For these reasons it is worth giving $\ker\varepsilon$ a special name. To this end, notice that the isotropy group $G_x$ is the same for every $x$ in the leaf $L$, and so is the group $\ker\varepsilon$.

\begin{definition}\label{dfn:essisotr}
Let $L$ be a leaf of the foliation $(M,\cF)$. The group $\ker\varepsilon$ is called the \emph{essential isotropy} group of $L$.
\end{definition}

Notice that the group $\mathring{H}_x^x$ is endowed with the quotient topology, which might be very coarse. In particular its identity element might \emph{not} form a closed subset of $\mathring{H}_x^x$ (an example of group with this property is $\R/\Q$). Hence $\mathring{H}_x^x$ might not be a Lie group, and from the continuity of $\varepsilon$ one can not conclude that the essential isotropy group is closed.   Of course, this issue arises only for singular leaves, since on a regular leaf the infinitesimal isotropy vanishes, whence the group $G_x$ is trivial.

 \begin{thm}\label{essclosed}
The identity component $\mathring{H}_x^x$ is smooth and is a Lie group if and only if the essential isotropy group is closed in $G_x$. {In this case, $\varepsilon \colon G_x \to H_x^x$ is a submersion.}\end{thm}
\begin{proof} Assume that the essential isotropy  is closed. Then it is an embedded normal Lie subgroup of $G_x$, whence it follows from Cartan's theorem \cite[Thm 4.2, p. 123]{Helgason} that the quotient $G_x / \ker\varepsilon$ is a Lie group and the projection $G_x \to G_x/\ker\varepsilon$ is a submersion. Since the sequence \eqref{seqgr} is exact, we conclude that  $\mathring{H}_x^x$ is a Lie group and $\varepsilon$ a submersion. 

Notice that the induced differentiable structure on $\mathring{H}_x^x$  is the one defined in \S \ref{subsubsection:Htop}. Indeed, for any  path holonomy bi-submersion $U$ minimal at $x$, we have that  $\tilde{G}_x$ is diffeomorphic to a neighborhood of $x$ in $U_x^x$, so $\sharp \colon U_x^x \to H_x^x$ is a submersion (see Lemma \ref{lem:inde}). \cite[Prop. 2.10 b)]{AndrSk} implies that the same holds for any bi-submersion $U$ in the path holonomy atlas covering $x$.

Conversely, if $\mathring{H}_x^x$ is a Lie group, then its identity element forms a closed subset, and since $\varepsilon$ is continuous it follows that the essential isotropy group is closed.
 \end{proof}

\begin{remark}
Going through examples, we couldn't find any foliation $(M,\cF)$ whose groups $\mathring{H}_x^x$ are not Lie groups for every $x \in M$. It may well be the case that the essential isotropy groups are always closed, however we couldn't provide a proof for this statement.
\end{remark}

Here is an example of the map $\varepsilon$:

\begin{cor}\label{ex:lin}
Let $V$ be a vector space and $G\subset GL(V)$ a closed connected subgroup, and consider the foliation on $V$ induced by the action. At $x=0$ we have that  $\varepsilon \colon G_x \to H_x^x=G$ is a covering map.
% integrating the Lie algebra morphism $Id_{\g}$.
\end{cor}
\begin{proof} Let $x=0$.
 Recall that  we have $\g_x=\g:=Lie(G)$, so in particular $G_x$ is the simply connected Lie group integrating $\g$. Further we have 
 $H_x^x\cong G$, and  the identification  is given by (the restriction to $\{x\}\times G$ of) a natural quotient map $\psi$ from the transformation groupoid $V \rtimes G$ to the holonomy groupoid $H$
(see   \cite[Prop. 1.4]{AndrSk} and \cite[Ex. 3.7(2)]{AndrSk}).
 
Consider the  natural minimal bi-submersion, namely $U \subseteq V \rtimes \g$, with target $\bt(y,u)=exp_y(u_V)$, where $u_V$ denotes the linear vector field on $V$ induced by the infinitesimal action of $u\in \g$. It is the path holonomy  bi-submersion associated to $\{(u_i)_V\}$ for $\{u_i\}$ any basis of $\g$. Notice that $U_x^x$, for $x=0$, is a neighborhood of the origin in $\g$.
One computes using eq. \eqref{GxU} that the map  $\Delta \colon \tilde{G}_x  \to U_x^x$ is given by the inverse of the Lie group exponential map $\g \to G_x$. Further the quotient map 
$\sharp \colon  U_x^x  \to H_x^x=G$ is the Lie group exponential map, as a consequence of the fact that the morphism of bi-submersions $\phi \colon U \to V \rtimes G, (y,u)\mapsto (y,exp_{\g}(u))$ intertwines the natural quotient maps
  $\sharp \colon  U   {\to} H$ and  $\psi \colon V \rtimes G\to H$ (see  the text just before \cite[Def. 3.5]{AndrSk}). Hence 
$$\tilde{\varepsilon} \colon \tilde{G}_x \stackrel{\Delta}{\to} U_x^x \stackrel{\sharp}{\to} H_x^x=G$$ is the local Lie group morphism that integrates $id_\g$.
%The map $$W \to V \rtimes G, \quad (y,u)\to (y,exp(u))$$ is a morphism of bisubmersions to the transformation groupoid of the action of $G$ on $V$ and a local diffeomorphism. By \cite[Ex. 3.7(2))]{AndrSk} there is a morphism of groupoids from the latter to $H$. Whence the composition $W \to H$ is just the restriction of the quotient map $\natural$. 
%We have $W_x^x = \tilde{\g}_x$ and our map $\tilde{E} \colon \tilde{\g}_x \to G$ is obtained precomposing with the map $\tilde{E}$, which is just $Id_{\g}$. So  $\tilde{E} \colon \tilde{\g}_x \to G$  is the exponential map.
\end{proof}

\subsection{\texorpdfstring{Smoothness of $H_L$}{Smoothness of $H_L$}}\label{subsection:longitsmooth}

Here we show that, under suitable assumptions, the restriction of $H$ to a leaf is a Lie groupoid.  
We start with  a technical lemma:

\begin{lemma}\label{splitmor} Let $(M,\cF)$ be a foliated manifold and $x$ a point lying in a leaf $L$ of dimension $k$.
As in Prop. \ref{transversal}, in a neighborhood $W$ of $x$, choose a splitting  $\psi \colon (W,\cF_W) \cong   (I^{k},TI^{k})\times (S,\cF_S)$  and
elements $\{X_i\}_{i\le n}$ of $\cF$  inducing a basis of $\cF_x$ and compatible with the splitting\footnote{In particular  the first elements $X_1,\dots,X_{k}$ are coordinate vector fields on $L$ commuting with all $X_i$'s, and the remaining  $\ell:=n-k$ elements $X_{k+1},\dots,X_n$ vanish on $L$.}.
Let $U\subset W\times \R^{n}$ be the associated path holonomy bi-submersion. Assume for simplicity that $U=W\times B^k \times B^{\ell}$ where $B^{k},B^{\ell}$
are open neighborhoods of the origin in $\R^k, \R^{\ell}$ respectively.

 Then for any  $y_0\in L_W:= \psi^{-1}(I^k\times \{x\})$, for any $\overrightarrow{\eta_0},\overrightarrow{\hat{\eta}_0}\in B^k$ and
$\overrightarrow{\lambda_0},\overrightarrow{\hat{\lambda}_0}\in B^{\ell}$ 
  we have:
$$(y_0,\overrightarrow{\eta_0},\overrightarrow{\lambda_0}) \sim (y_0,\overrightarrow{\hat{\eta}_0},\overrightarrow{\hat{\lambda}_0})
\Leftrightarrow 
\begin{cases}(x,0,\overrightarrow{\lambda_0}) \sim (x,0,\overrightarrow{\hat{\lambda}_0})\\
\overrightarrow{\eta_0}=\overrightarrow{\hat{\eta}_0}\end{cases}$$
\end{lemma}

\begin{remark}\label{lem:splitmor}
Lemma \ref{splitmor} is saying the the equivalence relation $\sim$ introduced in \S \ref{sec:groatlas}, restricted to  $U_{L_W}:=L_W\times B^k \times B^{\ell}$, is trivial in the $L_W$ and $B^k$ components, while in the $B^{\ell}$ component it is the restriction of $\sim$ to
 $U^x_x=\{x\}\times \{0\} \times B^{\ell}$.
\end{remark}

\begin{proof} Notice first that, since  any of the $\{X_i\}_{i \le k}$ commutes with any of the 
$\{X_i\}_{i > k}$, the target map sends $(z, \overrightarrow{\eta},\overrightarrow{\lambda})\in U$ to the element
\begin{equation}\label{formtar}
\big(\overrightarrow{z_I}+ \overrightarrow{\eta}\;,\;exp_z(\sum_{i>k}\lambda_i X_i)\big)\in I^k\times S=W,
\end{equation}
where $\overrightarrow{z_I}$ is the first component of $z\in W=I^k\times S$. 

Consider now an arbitrary morphism of bi-submersions $\Phi \colon U \to U$. From the first component of eq. \eqref{formtar} we deduce that 
%any morphism of bi-submersions $\Phi \colon U \to U$  
$\Phi$ is of the form
\begin{equation}\label{formf}
(z, \overrightarrow{\eta},\overrightarrow{\lambda})\mapsto (z, \overrightarrow{\eta}, \overrightarrow{f}(z, \overrightarrow{\eta},\overrightarrow{\lambda})),
\end{equation}
for some smooth $\overrightarrow{f} \colon U \to B^{\ell}$,
%i.e., $\Phi$ preserves not only the $W$-component but also the $B^k$-component. 
and from the second component of eq. \eqref{formtar} we deduce further that
\begin{equation}\label{expind}
exp_z(\sum_{i>k}\lambda_i X_i)=exp_z(\sum_{i>k}f_i(z, \overrightarrow{\eta},\overrightarrow{\lambda}) X_i)
\end{equation}
%showing that this expression is independent of $\overrightarrow{\eta}$.
for all $z\in W$ and $\overrightarrow{\lambda}\in B^{\ell}$ (in particular the r.h.s. is independent of $\overrightarrow{\eta}$). Conversely, any map $U\to U$ of the form \eqref{formf} satisfying eq. \eqref{expind} is a morphism of bi-submersions.

To prove the implication ``$\Leftarrow$'', assume that $(x,0,\overrightarrow{\lambda_0}) \sim (x,0,\overrightarrow{\hat{\lambda}_0})$. Then there exists a   morphism $\Phi$ defined near $(x,0,\overrightarrow{\lambda_0})$ mapping one point to the other, that is, the function $\overrightarrow{f}$ satisfies $\overrightarrow{f}(x,0,\overrightarrow{\lambda_0})=\overrightarrow{\hat{\lambda}_0}$.
Given $y_0\in L_W=\psi^{-1}(I^k\times \{x\})$ and $\overrightarrow{\eta_0}\in B^k$, define
$\Phi'$ in a neighborhood of $(y_0, \overrightarrow{\eta}_0,\overrightarrow{\lambda}_0)$ in $U$ by 
$$(z,  \overrightarrow{\eta},\overrightarrow{\lambda})\overset{\Phi'}{\mapsto} (z,  \overrightarrow{\eta}, \overrightarrow{f}(z-y_0, \overrightarrow{\eta}-\overrightarrow{\eta_0},\overrightarrow{\lambda}))$$
where $z-y_0=(z_I-y_0,z_S)\in I^k\times S=W$. 
The map $\Phi'$ is a morphism of bi-submersions for it satisfies eq. \eqref{expind}, as a consequence of the fact that $\Phi$ satisfies eq. \eqref{expind} and of the invariance of the vector fields $X_i$ in the $I^k$-direction. Further $\Phi'$ maps
$(y_0,\overrightarrow{\eta_0},\overrightarrow{\lambda_0})$ to $(y_0,\overrightarrow{\eta_0}, \overrightarrow{\hat{\lambda}_0})$, showing that these two elements are equivalent. The converse implication is proven similarly.
\end{proof}
 
 \begin{thm}\label{HL} Let $L$ be a leaf and $H_L \rightrightarrows L$ the transitive   groupoid $H_L:=\bs^{-1}(L)= \bt^{-1}(L)$. The essential isotropy group (Def. \ref{dfn:essisotr}) of $L$ is closed if and only if
$H_L$ is smooth and is a Lie groupoid.
\end{thm}
\begin{proof} If $H_L$ is a (transitive) Lie groupoid then $H_x^x$ is a Lie group, so Thm. \ref{essclosed} implies that the essential isotropy group of $L$ is closed. 

For the converse, we show that the assumptions of Prop. \ref{prop:allsmooth} are satisfied, and then apply Lemma \ref{lem:smooth}. Let $x\in L$, let $U$ be a bi-submersion as in Lemma \ref{splitmor}, and set $U_{L_W}:=L_W\times B^k \times B^{\ell}$. 
%We show that $\sharp U_{L_W}$ provides  a manifold chart for $H_L$ near the point $x$.
  
Consider first $U_x^x$. {Thm. \ref{essclosed}} 
%Cor. \ref{subm} 
implies that   
$\sharp|_{U_x^x} \colon U^x_x \to H^x_x$ is a submersion between smooth manifolds.
Let $S^x\subset U_x^x$ be  a  smooth submanifold through $x$ transverse to the fibres of 
$\sharp|_{U_x^x}$, and without loss of generality assume that $\sharp S^x=\sharp U_x^x$ (an open subset  of $H_x^x$,
 by Lemma \ref{openim}).
%Since the quotient $U_x^x/\sim=H_x^x$ is smooth by Cor. \ref{cor:integr}, it follows that 
Clearly $\sharp|_{S^x}\colon S^x \to \sharp U_x^x$ is  a diffeomorphism.

 Notice that $\sharp U_{L_W}$ is an open neighborhood of $x$ in  $H_L$, since $\sharp U$ is open in $H$ ({see the beginning of  \S \ref{subsubsection:Htop}}).  
 Since by definition 
the fibers of $\sharp \colon U_{L_W} \to H_L$ are given by the equivalence classes of $\sim$ in $U_{L_W}$, 
Lemma \ref{splitmor} (see also Remark \ref{lem:splitmor}) implies that at every point $(y,\overrightarrow{\eta},\overrightarrow{\lambda})$ of $U_{L_W}$ we have
%$$U_x/\sim
\begin{equation*} \text{(Fibre of $\sharp \colon U_{L_W} \to H_L$ through $(y,\overrightarrow{\eta},\overrightarrow{\lambda}$))
 $= \{y\}\times \{\overrightarrow{\eta}\} \times $(Fibre of $\sharp|_{U_x^x}$ through $(x,0,\overrightarrow{\lambda}) )$}. 
\end{equation*}  
From this we conclude that $T^x:=L_W\times B^k\times S^x$ is a smooth submanifold of $U_{L_W}$ through $(x,0,0)$ transverse to the $\sharp$-fibres.  Therefore
$\sharp \colon U_{L_W} \to H_L$ restricts to a homeomorphism
$$\sharp|_{T^x} \colon T^x \to \sharp U_{L_W},$$
% whose inverse provides 
which endows $\sharp U_{L_W}$ with a differentiable structure such that that $\sharp \colon U_{L_W} \to H_L$ is a submersion. As $x\in L$ is arbitrary, we showed that the assumptions of Prop. \ref{prop:allsmooth} are satisfied.
%
%The differentiable manifold structure near $x$ is independent of the choice of path holonomy bi-submersion $U$, since any two such bi-submersions are (smoothly) isomorphic near $x$ by Lemma \ref{lemma:unique2}.  Even more,  \cite[Prop. 2.10 b)]{AndrSk} implies that for any bi-submersion $U$ in the path holonomy atlas covering $x$, the quotient map $U_L \to H_L$ is a submersion nearby $x$.
%Notice that the source and target maps of $H_L$ restricted to the open subset $\sharp U_{L_W}$ are submersions, as the same holds for the restriction to $T^x$ of the source and target maps of $U_{L_W}$ (use eq.\eqref{formtar}), and since $\sharp$ preserves the source and target maps. From this we conclude that there is a neighborhood of $L$ in $H_L$ which is a local Lie groupoid.
\end{proof}

%It is well-known that transitive Lie groupoids correspond to principal bundles, that is, they are of the form $(P\times P)/G \rightrightarrows M$ for a $G$-principal $P$ bundle over $M$. Hence from Thm. \ref{source} we obtain:

\begin{prop}\label{source} Let $x\in M$ and $L$ the leaf of $\cF$ at $x$. Its essential isotropy group is closed if and only if the source fiber $H_x$ of the holonomy groupoid is smooth (in the sense of \S \ref{subsubsection:Htop}). In this case, 
 $\bt \colon H_x \to L$ is a principal $H_x^x$-bundle.
\end{prop} 
 \begin{proof}
As we showed in Thm. \ref{HL}, if the essential isotropy is closed then $H_L$  is  smooth, whence $H_x$   is smooth as well. For the converse, it suffices to show that $H_x^x$ is smooth, because it follows then from Thm. \ref{essclosed} that the essential isotropy group is closed.

Let $(U,\bt,\bs)$ be a  bi-submersion in the path holonomy atlas with $x\in \bs(U)$. Put $U_x = \bs^{-1}(x)$ and $\bt_x = \bt\mid_{U_x} : U_x \to L$. The latter is a   submersion, as a consequence of the definition of bi-submersion. Hence $U_x^x=\bt_x^{-1}(x)$ is an embedded submanifold of $U_x$ (if it is not empty). By assumption, $H_x$ is endowed with a smooths structure such that $\sharp \colon U_x \to H_x$ is a submersion.
$U_x^x$ is a union of   fibers of   $\sharp$, by the definition of the equivalence relation $\sim$. This and the local form theorem for submersions imply that $\sharp U_x^x$ is a submanifold of $H_x$ and that $\sharp|_{U_x^x} \colon U_x^x \to \sharp U_x^x$ is a submersion. As the bi-submersion $U$ was arbitrary, it follows that $H_x^x$ is smooth in the sense of \S \ref{subsubsection:Htop}.

To prove the second statement, use Thm. \ref{HL} and recall that
any source fibre of a Lie groupoid is a principal bundle for the corresponding isotropy group \cite[Thm. 5.4]{MM}.
%To see this, consider a point $x \in M$, let $L$ be the leaf at $x$ and $(U,\bt,\bs)$ the path holonomy bi-submersion at $x$. Put $U_x = \bs^{-1}(x)$, $\bt_x = \bt\mid_{U_x} : U_x \to L$. Since $\bt_x$ is a submersion it has a smooth local section, namely there exists an open $W \subseteq L$ and a smooth $\hat{\sigma} : W \to U_x$ which is right-inverse to $\bt_x$. Consider the quotient map $\sharp : U_x \to H_x$ and put $\sigma := \sharp \circ \hat{\sigma}$. This is a continuous local right-inverse for the target map of $H_x$. The trivializing map $$W \times H_x^x \to \sharp U_x,\quad (y,h)\mapsto \sigma(y) h$$ is easily seen to be a homeomorphism. Due to the discussion in \S \ref{subsubsection:Htop} an open subset of $H_x$ is of the form $W\times H_x^x$. Whence, having $H_x$ smooth forces $H_x^x$ to be smooth as well. In this case, the trivialization above shows that $H_x$ is a principal bundle. 
\end{proof}

Proposition \ref{source} and Thm.  \ref{ALintegr} are crucial for the generalization of Heitsch's \cite[Thm. 3]{Heitsch} to singular foliations. We will explain this in a subsequent article \cite{AZ2}.

\section{Essential isotropy and integrability}\label{section:integr} 
We prove  that the Lie algebroid $A_L$ is integrable provided the essential isotropy group of the leaf $L$ is discrete (\S \ref{subsection:ALintegr}). Then, in \S \ref{subsection:secordvf} we provide a sufficient condition for the discreteness of the essential isotropy group.
%,  is the local surjectivity of the exponential map associated with a certain Lie algebra of ``second order'' transversal vector fields. 
%We end the section and the sequel by outlining in \S \ref{subsection:outline} the role of this Lie algebra in the understanding of the transversal action of the holonomy groupoid, and the proof of a normal form theorem for singular foliations. This will be discussed thoroughly in a subsequent article.

\subsection{\texorpdfstring{Integrability of $A_L$}{Integrability of AL}}\label{subsection:ALintegr}

Fix a leaf $L$ and  consider the associated transitive Lie algebroid, $A_L=\cup_{x\in L} \cF_x$, 
%whose sections contain $\cF/I_L\cF$
see \S \ref{subsection:ALintro}. Here we discuss the integrability of this Lie algebroid, and show that the integrability of $A_L$ implies the smoothness of the $\bs$-fibers of the holonomy groupoid. When the $\bs$-fibers are smooth, then the construction of the pseudodifferential calculus in \cite{AndrSk1} is simplified a great deal because then the left-regular representation is defined.

\begin{thm}\label{ALintegr}
Let $(M,\cF)$ be a foliation and $L$ a leaf. The transitive groupoid $H_L$ is smooth and integrates the Lie algebroid $A_L   = \cup_{x\in L}\cF_x$ if and only if the essential isotropy group of $L$  (Def. \ref{dfn:essisotr}) is discrete.  
\end{thm}
\begin{remark}
{Notice that the Lie algebroid $A_L$ integrates to $H_L$ if{f} the Lie algebra $\g_x$ integrates to $\mathring{H}_x^x$ (for any $x\in L$), for both conditions are equivalent to the discreteness of the essential isotropy group of $L$.}
\end{remark}
\begin{proof} It is known that every Lie algebroid integrates to a \emph{local} Lie groupoid.
We start by describing explicitly a local Lie groupoid that integrates $A_L$. Recall from \S \ref{subsection:fol} that the leaf $L$ is endowed with the longitudinal smooth structure. Cover $L$ by open subsets $W_{\alpha}$ of $M$ as in the splitting theorem (Prop. \ref{transversal}), i.e. so that there exists an isomorphism of foliated manifolds
$\psi_{\alpha} \colon (W_{\alpha},\cF_{W_{\alpha}}) \cong (I^k,TI^k)\times (S_{\alpha},\cF_{S_{\alpha}})  $ where $I=(-1,1)$ and $S_{\alpha}$ is a slice transverse to $L$ at some point $x_{\alpha}$. We  fix a choice of  generators $\{X_i^{\alpha}\}_{1\le i \le n} $  of $\cF_{W_{\alpha}}$  as in Lemma \ref{splitmor}. Denote $L_{\alpha}:=\psi_{\alpha}^{-1}(I^k\times \{x_{\alpha}\})$.
%with the following properties: all of the them commute with $X_i^{\alpha}$ for all $i\le k$, the restrictions of $X_{1}^{\alpha},\ldots,X_k^{\alpha}$ to $L_{\alpha}:=W_{\alpha}\cap L$ form a frame of $T{L_{\alpha}}$ and the restrictions of  $X_{k+1}^{\alpha},\ldots,X_n^{\alpha}$ to $L_{\alpha}$ vanish. 

The vector fields $\{X_i^{\alpha}\}_{1\le i \le n}\subset \cF$ induce a frame   $\{\xi_i^{\alpha}\}_{1\le i \le n} $ of $A_L|_{L_{\alpha}}$ via the projection  $\cF_{W_{\alpha}} \to C^{\infty}(L_{\alpha},A_L|_{L_{\alpha}})=\cF_{W_{\alpha}}/I_{L_{\alpha}}\cF_{W_{\alpha}}$. This frame induces an isomorphism of Lie algebroids over $L_{\alpha}$
\begin{equation*}
\nu_{\alpha} \colon A_L|_{L_{\alpha}} \cong TL_{\alpha}\oplus \g_x
\end{equation*}
which maps $\xi_i^{\alpha}\mapsto X_i^{\alpha}|_{L_{\alpha}}$ for $i\le k$ and $\xi_i^{\alpha}
\mapsto [X_i^{\alpha}]\in \g_x$ for $k< i \le n$. A Lie groupoid integrating $TL_{\alpha}\oplus \g_x\to L_{\alpha}$ is $(L_{\alpha}\times L_{\alpha}) \times G_x$, the product of the pair groupoid over $L_{\alpha}$ and the simply connected Lie group integrating the Lie algebra $\g_x$. We will now relate this Lie groupoid   the  path holonomy bi-submersion $U_{\alpha}$ defined by the $\{X_i^{\alpha}\}_{1\le i \le n}\subset \cF$ . 

As in Lemma \ref{splitmor}, assume that $U_{\alpha}$ is of the form $W_{\alpha}\times B^{k} \times B^\ell$ where $B^k,B^{\ell}$
are open neighborhoods of the origin in $\R^k,\R^{\ell}$ respectively. We make use  of $\R^{\ell}\cong \g_x$  given by the basis $\{[X_i^{\alpha}]\}_{k < i \le n}$  of $\g_x$, and we use the splitting theorem to identify $L_{\alpha}$ with $I^k$.
The identification 
\begin{align}\label{ideOm}
L_{\alpha} \times   B^k \times B^{\ell} &\to
\text{neighborhood of the identity section of }(L_{\alpha}\times L_{\alpha}) \times G_x,\\
 (z,\vec{\eta},\vec{\lambda})&\mapsto (z+\vec{\eta},z, exp_{\g_x} \vec{\lambda})\nonumber
\end{align}
makes    $L_{\alpha} \times   B^k \times B^{\ell}$ into a local Lie groupoid, which we denote by $\Omega_{\alpha}$. It has the following properties: $\Omega_{\alpha}$
 integrates  $A_L|_{L_{\alpha}}$ (because $(L_{\alpha}\times L_{\alpha}) \times G_x$ does), and its
  source and target map of $\Omega_{\alpha}$   agree with those of the path holonomy bi-submersion $U_{\alpha}$ by 
eq. \eqref{formtar}.

Consider the restriction $ \Omega_\alpha|_{L_{\alpha\beta}}$   of $\Omega_\alpha$ to $L_{\alpha\beta} := L_{\alpha} \cap L_{\beta}$. 
%Also put $\mu_{ij} = \mu_j^{-1} \circ \mu_i : \Omega_{ij} \to \Omega_{ji}$. 
%\Omega_{\alpha}
The Lie algebroid isomorphisms $$\nu_{\beta}\circ \nu_{\alpha}^{-1} \colon (TL_{\alpha}\oplus \g_x)|_{L_{\alpha\beta}}\to (TL_{\beta}\oplus \g_x)|_{L_{\alpha\beta}} $$ integrate to isomorphisms of local Lie groupoids $\mu_{\beta \alpha} \colon \Omega_\alpha|_{L_{\alpha\beta}} \to \Omega_\beta|_{L_{\alpha\beta}}$
 which satisfy $$\mu_{\alpha\alpha}=id, \quad \mu_{\alpha\beta}^{-1} = \mu_{\beta\alpha}, \quad \mu_{\alpha\beta}\circ \mu_{\beta\gamma} = \mu_{\alpha\gamma}.$$
  Now define the equivalence relation $$\Omega_{\alpha} \ni \omega_\alpha \sim \omega_\beta \in \Omega_{\beta} \quad \Longleftrightarrow \quad \omega_{\alpha} \in \Omega_\alpha|_{L_{\alpha\beta}}, \quad \omega_\beta\in \Omega_\beta|_{L_{\alpha\beta}} \quad \text{and} \quad \mu_{\beta \alpha}(\omega_{\alpha}) = \omega_\beta$$ Then the local Lie groupoid $\Omega :=\coprod_{{\alpha}} \Omega_{\alpha} /\sim$ is  a local Lie groupoid integrating $A_L$.  
    
For every $\alpha$, consider  the quotient map $\sharp \colon \Omega_{\alpha} \to H_L$, and recall that $\Omega_{\alpha}$ is a local Lie groupoid whose underlying bi-submersion is $U_{\alpha}$.
Since the $\mu_{\beta \alpha}$ are in particular morphisms of bi-submersions, by the definition of holonomy groupoid 
% is morphism of local Lie groupoids.
these quotient maps  assemble to a map $\phi \colon \Omega  \to H_L$. 

We  claim: \emph{$\phi \colon \Omega  \to H_L$ is a morphism of local groupoids}.

To prove the claim it suffices  to show that the quotient map $\sharp \colon \Omega_{\alpha} \to H_L$ is a morphism of local groupoids, for every $\alpha$.
In turn, this is a consequence of the fact that 
the multiplication and inversion of any local Lie groupoid are morphisms of bi-submersions.  More explicitly:  $\sharp \colon \Omega_{\alpha} \to H_L$  preserves multiplications because in the diagram
\begin{equation*}
\xymatrix{
\Omega_\alpha \circ \Omega_\alpha \ar[d]^{\sharp \times \sharp}\ar[dr]^{\sharp|_{\Omega_\alpha \circ \Omega_\alpha}}  \ar[r]^{mult.} & \Omega_\alpha \ar[d]^{\sharp} \\
(H_L) _{\bs}\times_{\bt} (H_L) \ar[r]^{\;\;\;\;\;\;\;\;\;mult.}& {H_L} }
\end{equation*}
the lower triangle commutes by the definition of multiplication on $H$ and the upper triangle commutes since the multiplication on  $\Omega_{\alpha}$ is a  morphism of bi-submersions. This proves the claim.

Notice that $\phi$ maps onto an open neighborhood of $L$ in $H_L$, as a consequence of the fact that  $\sharp U$ is open in $H$ for any path holonomy bi-submersion $U$ (see the beginning of  \S \ref{subsubsection:Htop}).
At every point $x\in L$, by eq. \eqref{tvare},    in a neighborhood of the unit in $G_x$,
the essential isotropy group is identified with $\phi^{-1}(x)$.
\emph{Hence the  essential isotropy group of the leaf $L$
is discrete if{f} the fibers of $ker(\phi)=\cup_{x\in L}\phi^{-1}(x)$  are discrete}.

Assume now that the essential isotropy group
is discrete.  Restricting the extension $$0 \to \ker\phi \to  \Omega \to H_L \to 0$$
to a neighborhood of the set of units $L$ we deduce that a neighborhood of $L$ in $H_L$ is a local Lie groupoid integrating $A_L$.  
% Proceeding as in the last paragraph of Thm. \ref{HL} 
Since $\Omega$ is a union of   path holonomy bi-submersions, from Lemma \ref{lem:smooth} and Prop. \ref{prop:allsmooth} we deduce that the   $H_L$ is smooth and is a Lie groupoid.  Conversely, if we assume that $H_L$ is a Lie groupoid integrating $A_L$, then the above extension shows that 
$ker(\phi)$ is discrete.
\end{proof}
   
 \begin{remark}\label{MRui}
Crainic and Fernandes \cite{CrF} determined the exact obstruction for the integrability of the Lie algebroid $A_L$ to a \emph{source simply connected} Lie groupoid $\Gamma$: it is the discreteness (or equivalently, the closeness) of a certain subgroup $\tilde{\cN}_x(A_L)$ of the center of $G_x$, called  \emph{monodromy group}. 

When the essential isotropy group $\ker\varepsilon$ is discrete then the Lie algebroid $A_L$ is integrable by Thm. \ref{ALintegr}, and we have $$\tilde{\cN}_x(A_L)\subset \ker\varepsilon.$$ Indeed, $\tilde{\cN}_x(A_L)$ is defined as the kernel of the morphism $G_x \to \Gamma_x^x$ induced by the inclusion  ${\g_x} \hookrightarrow A_L$, there is a morphism $\Gamma_x^x \to H_x^x$ induced by (the restriction of) $Id_{A_L}$, and
 $\varepsilon$  factors as $G_x \to \Gamma_x^x \to H_x^x$ (since $\varepsilon$  is induced by $Id_{\g_x}$). 
% where $\Gamma$ is the source simply connected Lie groupoid integrating $A_L$.
{At this stage} we do not know how  $\ker\varepsilon$ relates to the monodromy group in general. For instance, it is not clear if, in the general case, the discreteness of  $\tilde{\cN}_x(A_L)$ implies the discreteness of $\ker\varepsilon$, for the former controls the integration of $A_L$ to a source simply connected Lie groupoid, whereas the latter controls the integration of $A_L$ to $H_L$. {The general comparison of the two integrability obstructions is beyond the scopes of this work; this might be discussed in a different paper.}

Notice that, whenever  $\ker\varepsilon$ is totally disconnected (for instance, discrete) it automatically  contained in the center of $G_x$. Indeed if $k\in \ker\varepsilon$ and $g\in G_x$, consider a continuous path $t \mapsto g_t$ in $G_x$ from the identity to $g$. The path $t \mapsto g_t k g_t^{-1}$ is contained in $\ker\varepsilon$   and sits at $k$ at time zero, so the fact that   $\ker\varepsilon$ is totally disconnected implies that $g k g^{-1}=k$, i.e. $k$ commutes with $g$.\end{remark}  
   
%   \mc{Should we leave the next remark?}
%\begin{remark}
%Recall the following definition from  \cite[Rem. 3.13]{AndrSk}: the holonomy groupoid $H$ is \emph{longitudinally smooth} if for any $x\in M$ there exists a bi-submersion $U$ in the path holonomy atlas, a point $u\in U_x^x$ carrying the identity, and an open neighborhood $\tilde{U}$ of $u$ in the source fiber $U_x$ such that $\sharp \colon \tilde{U} \to H$ is injective. If the isotropy group of every leaf of $(M,\cF)$ is discrete, then $H$ is longitudinally smooth (see the proofs  of  Thm. \ref{ALintegr}).
%\end{remark}
   
\begin{remark} The  structure of the local Lie groupoid $\Omega_{\alpha}=L_{\alpha} \times   B^k \times B^{\ell} \rightrightarrows W_{\alpha}$ induced by the identification \eqref{ideOm} is the following\footnote{Recall that we use the splitting theorem to identify $W_{\alpha}$ with $I^k$}:
\begin{description}
\item[-]  $\bs(z,\vec{\eta},\vec{\lambda})=z$ and $\bt(z,\vec{\eta},\vec{\lambda})=z +\vec{\eta}$;
\item[-] The unit map is the embedding of $W_\alpha$;
\item[-] The local product is given by $$(z_1,\vec{\eta_1},\vec{\lambda_1})\cdot (z_2,\vec{\eta_2},\vec{\lambda_2}) = (z_2,\vec{\eta_1}+\vec{\eta_2},BCH(\vec{\lambda_1},\vec{\lambda_2}))$$ where $BCH(\vec{\lambda_1},\vec{\lambda_2})$ stands for the Baker-Campbell-Hausdorff product of $\vec{\lambda}_i \in B^\ell\subset \g_x$;
%$\vec{\lambda} \in B^{\ell}$ such that $BCH(\sum_{i\leq\ell}\lambda_1^i [X_i],\sum_{i\leq\ell}\lambda_2^i [X_i]) = \sum_{i\leq \ell} \lambda^i [X_i]$. (Note that since $X_1,\ldots,X_\ell$ vanish on $L$, they vanish at $x$, whence $\sum_{i\leq \ell}\lambda_i^{\tau}[X_i] \in \g_x$ for $\tau = 1,2$.)
\item[-] The inversion map is $$\iota : g=(z,\vec{\eta},\vec{\lambda}) \mapsto (z+\vec{\eta},-\vec{\eta},-\vec{\lambda}).$$

%which is defined for those $((y,\vec{\lambda}),(x,\vec{\mu})) \in G \times G$ such that $y = t(x,\vec{\mu})$ and $(x,\vec{\mu}) \in Dom(exp(\sum \lambda_i Y_i))$ and $exp_{(x,\vec{\mu})}(\sum \lambda_i Y_i) \in G$
\end{description}
\end{remark}

\begin{ep}  Let $X$ be a complete vector field on $M$ and $\cF=\langle X \rangle$. As in Ex. \ref{eps:linact} (ii) assume that, for all $x\in \partial \{X=0\}$, every neighborhood of $x$ contains at least one point   the integral curve through which is not periodic. Fix such a point $x$ and denote $L=\{x\}$.  Then  $\varepsilon \colon G_x \to H_x^x$ 
is a surjective group morphism $(\R,+)\to (\R,+)$, i.e. a non-zero multiple of $Id_{(\R,+)}$, hence $ker(\varepsilon)=\{0\}$
Clearly $H_L=(\R,+)$ integrates the Lie algebroid $A_L=\g_x=\R$, as predicted by Thm. \ref{ALintegr}.
\end{ep}

\begin{ep} Let $V$ be a finite dimensional vector space and  $G\subset G(V)$ a closed connected subgroup. Consider the singular foliation given by the action of $G$ on $V$. Let $L=\{x\}$ for $x=0$. Then  $\varepsilon \colon G_x \to H_x^x=G$ is a covering map by Cor. \ref{ex:lin}, where $G_x$ denotes the simply connected cover of $G$. Hence $ker(\varepsilon)$ is discrete.
Clearly $H_L=H_x^x=G$ integrates the Lie algebroid $A_L=\g_x=\g$, as predicted by Thm. \ref{ALintegr}. 
\end{ep}

\begin{ep}
Consider $M=\T^2\times \R^2$, with coordinates $(\theta_1,\theta_2,t_1,t_2)$, endowed  with the singular foliation $\cF$ spanned by 
\begin{align*}
&v_1=\p_{\theta_1}+t_1t_2\p_{t_1}\;\;\;\;\;\;  w_1= t_1^2t_2\p_{t_1}  \\ 
&v_2=\p_{\theta_2}+t_1t_2\p_{t_2}\;\;\;\;\;\;  w_2= t_1t_2^2\p_{t_2}.  
\end{align*}
$\cF$ is involutive, as $$[v_1,v_2]=w_2-w_1,\;\;\; [v_1,w_1]=t_2w_1,\;\;\; 
[v_1,w_2]=t_2(w_2-w_1),\;\;\;  [w_1,w_2]=t_1t_2(w_2-w_1).$$
(The remaining brackets are deduced easily from the symmetry that relates $v_1$ to $v_2$ and $w_1$ to $w_2$.)

For $L=\T^2\times\{0\}$ the Lie algebroid $A_L$, as a vector bundle, is the trivial rank 4 vector bundle. The images of the above vector fields under the quotient map $\cF\to \cF/I_L\cF=C^{\infty}(L;A_L)$ form a trivializing frame of sections $\underline{v_1},\underline{v_2},
\underline{w_1},\underline{w_2}$ of $A_L$. The Lie algebroid structure of $A_L$ in terms of this frame is as follows: the anchor maps $\underline{v_i}\mapsto \p_{\theta_i}$  and $\underline{w_i}\mapsto 0$  , while the only non-zero bracket is $[\underline{v_1},\underline{v_2}]_{A_L}=\underline{w_2}-\underline{w_1}$. 

The Lie algebroid $A_L$ is integrable, as every transitive Lie algebroid over a base with trivial second homotopy group $\pi_2$ \cite[Cor. 5.5]{CrF}. Explicitly, denote by $\mathfrak{k}$ the 4-dimensional Lie algebra whose basis we denote (abusing slightly   notation) by $\underline{v_1},\underline{v_2},
\underline{w_1},\underline{w_2}$, satisfying the above bracket relations. 
$A_L$ is the transformation algebroid for the action of  $\mathfrak{k}$ on $L$ given by the above anchor map, so a Lie groupoid integrating $A_L$ is the transformation groupoid $K \ltimes L \rightrightarrows L$ where $K$ is the simply connected Lie group integrating $\mathfrak{k}$. Notice that $K$ is the product of 
$\left\{\left(\begin{smallmatrix}  1 & a & b\\
0 &1 & c\\
0&0& 1 \end{smallmatrix}\right): a,b,c \in \R \right\}$ with $(\R,+)$,
since  $\mathfrak{k}$ is the direct sum of the Heisenberg Lie algebra with $\R$.
 \end{ep}

\begin{ep}
Let $L$ be a manifold and $\alpha\in \Omega^1(L)$ a closed 1-form. On
$M:=L\times \R$ consider the singular foliation spanned by $$w:=t^2\p_{t} \;\;\text{   and   }\;\;X^{\alpha}:=X+\alpha(X)t\p_{t},$$ as $X$ ranges over all vector fields of $L$ (extended trivially to vector fields on $L\times \R$). 
 Here $t$ denotes the coordinate on $\R$.
The singular foliation is involutive: $[X^{\alpha},w]=\alpha(X)w$, and $[X^{\alpha},Y^{\alpha}]=[X,Y]^{\alpha}$ since $\alpha$ is closed.
 
The Lie algebroid $A_L$ is isomorphic as a vector bundle to $TL\times \R \to L$ with the obvious anchor and bracket given by $[(X,0),(Y,0)]_{A_L}=([X,Y],0)$ and 
$[(X,0),(0,1)]_{A_L}=(0,\alpha(X))$. This Lie algebroid is integrable. Indeed, the formula $\nabla_X 1=\alpha(X)$ determines a flat connection on the trivial $\R$-bundle over $L$, that is, a representation of $TL$ on the trivial $\R$-bundle. The above Lie algebroid is precisely the transformation algebroid of this representation, and as such it is integrable (by the transformation groupoid induced by the holonomy of $\nabla$).
\end{ep}

 \subsection{A criterion for the discreteness of the essential isotropy group}\label{subsection:secordvf}
%The Lie algebra of ``second order'' transversal vector fields vanishing at a point

This section provides a condition for the discreteness of $ker(\varepsilon)$. 

\begin{thm}\label{thm:secord}  Let $S_x$ be a slice to the leaf $L_x$ at $x$. 
% Let $G_{2}(x)$ be the subgroup of $Diff(S_x)$corresponding to the Lie subalgebra $I_x\cF_{S_x}$ of $\cV(S_x)$. If every element of $G_{2}(x)$ sufficiently close to the identity is of the form $exp(Z)$ with $Z \in I_x\cF_{S_x}$ then the kernel of $$\sharp \colon U_x^x  {\to} H_x^x$$ is discrete.
Assume that the following condition is satisfied: for any smooth time-dependent vector field  $\{X_t\}_{t\in [0,1]}$ in $I_x\cF_{S_x}$, there exists a vector field $Z\in I_x\cF_{S_x}$ and 
a neighborhood $S'$ of $x$ in $S_x$ such that $exp(Z)|_{S'}=\phi|_{S'}$, where $\phi$ denotes the time-1 flow of $\{X_t\}_{t\in [0,1]}$.

Then the essential isotropy group of $L_x$ (Def. \ref{dfn:essisotr}) is discrete.  
\end{thm}

\begin{remark}\label{issueds}
$I_x\cF_{S_x}$ is a Lie subalgebra of the infinite dimensional Lie algebra $\vX_c(S_x)$, and it is well-known that 
the exponential map (the time-1 flow) $\vX_c(S_x) \to \text{Diff}(S_x)$
is \emph{not} locally surjective \cite[Ex.6.6]{Milnor} \cite[Ex. 43.2]{KrMi97}.
%Koppel works on $S^1$, Grabowski on [0,1]
The above condition amounts to a local surjectivity statement for the restriction of the exponential map  to $I_x\cF_{S_x}$ and to a subgroup of $\text{Diff}(M)$, where one is  allowed to shrink $S_x$ to a small neighborhood  of $x$. We do now know how restrictive the above condition is. 
\end{remark}

% that subalgebra $I_x\cF_{S_x}$ comprises of ``second-order'' vector fields in $\cF_{S_x}$, namely vector fields which locally have products of two functions as coefficients. This justifies our notation $G_{2}(x)$ for the corresponding group. 
%Also note that the local surjectivity of the exponential map is satisfied %for finite-dimensional Lie groups/algebras, but in the infinite dimension %there may be problems. 
\begin{proof}
Let $\{X_i\}_{i\le n}$ be generators of $\cF$ in a neighborhood of $x$, chosen compatibly with a splitting $W\cong I^k\times S_x$ as in Lemma \ref{splitmor} {(so $X_{k+1},\dots,X_n$ are tangent to $S_x$)}.
Denote by $U \subset W \times \R^n$ the corresponding path holonomy bi-submersion.
%(here $n=l+k$ where $\ell=\dim(\g_x)$ and $k=\dim(L_x)$).
We have $U_x^x=\{x\}\times \{0\} \times \R^\ell$ by eq. \eqref{formtar}.
Take $D\subset \R^\ell$ to be a neighborhood of the origin such that 
the matrix exponential $\mathfrak{gl}(T_xS_x)\to GL(T_xS_x)$ is injective on $\{\sum_{i> k}\gamma_i X_i^{lin}: \gamma \in D\}$, where $X_i^{lin}\in \vX_{lin}(T_xS_x)\cong \mathfrak{gl}(T_xS_x)$ denotes the linearization of $X_i$ at $x$.
We show that $\ker(\varepsilon)\cap {\Delta^{-1}(\{x\}\times \{0\}\times D)}$ consists of just one point, namely the identity element of $G_x$, see  eq. \eqref{tvare}.

Let $u:=(x,0,\lambda^0)\in U_x^x$ with $\lambda^0\in D$. Assume that $\Delta^{-1}(u)\in \ker(\varepsilon)$.  By Lemma \ref{lem:ker} this is equivalent to the existence of
  a bisection $s$ of $U$ through $u$ (say given by $z\mapsto (z, \eta(z),\lambda(z))$) carrying the identity. Notice that  in particular $\lambda(x)=\lambda^0$. We want to show that $\lambda^0=0$.
% can not be arbitrarily close to the origin of $\R^k$ unless it is the origin, 
%(This is sufficient, see Lemma \ref{lem:ker}.)

The diffeomorphism carried by  $s$ is (essentially) given 
by eq. \eqref{formtar}, so that
we necessarily have $\eta(z)=0$ for all $z$. From now on we consider only $U|_{S_x}$. Restricting the identity diffeomorphism to the slice $S_x$ we obtain that
$$S_x\to S_x, \quad z \mapsto exp_z(\sum_{i> k}\lambda_i(z) X_i)$$  is $Id_{S_x}$. Now consider the vector field $$Y:=\sum_{i> k}\lambda^0_i X_i|_{S_x}$$ on $S_x$. Denote by $\phi=exp(Y)$ its flow, which is also the diffeomorphism carried by the \emph{constant bisection} $z \mapsto (z,0,\lambda^0)$ of $U|_{S_x}$. We apply a theorem of \cite{AZ2}
 %Thm. \ref{globalaction} 
 to the foliation $(S_x,\cF_{S_x})$ (of which $\{x\}$ is a singular leaf)   at $u\in U_x^x$ in order to compare the diffeomorphisms carried by the bisection $s|_{S_x}$ (which is $Id_{S_x}$) and the constant bisection $\lambda^0$, and we deduce that  $\phi$ is the time-1 flow of a \emph{time-dependent} vector field $\{X_t\}_{t\in [0,1]}$ in $I_x \cF_{S_x}$.
 %$\phi\in\exp(I_x \cF_{S_x})$, 
%
The condition we assumed    implies that there is $Z\in I_x \cF_{S_x}$  such that $exp(Z)=\phi=exp(Y)$, {shrinking $S_x$ if necessary.}
Hence \begin{equation}\label{YX0}
Id_{S_x}=exp(Y)exp(-Z)=exp(BCH(Y,Z)). 
\end{equation}
Now consider the vector field $BCH(Y,Z)$. 
Notice that $BCH(Y,Z)=Y+Z'$ for some $Z'\in I_x \cF_{S_x}$.
The flow $\varphi^t:=exp(t(Y+Z'))$ fixes $x$.
The flow of the linearization $Y^{lin}=(Y+Z')^{lin}$ is the linearization of the flow of $Y+Z'$, which is $d_x\varphi^1=d_xId_M=Id_{T_xM}$. Since the flow of linear vector fields on $T_xS_x$ is given by the matrix exponential $\mathfrak{gl}(T_xS_x)\to GL(T_xS_x)$ and  $\lambda^0\in D$, the injectivity condition on $D$ implies that $(Y+Z')^{lin}=0$. Hence its flow $d_x \varphi^t$ is the identity for any $t$. Bochner's linearization theorem  \cite[Thm. 2.2.1]{DK} implies that the circle\footnote{We have an $S^1$ action because of eq. \eqref{YX0}.} actions on $S_x$ (by $t \to \varphi^t$) and on $T_xS_x$ (by $t\mapsto d_x \varphi^t=Id$) are equivalent, hence we conclude that $\varphi^t=Id_M$ for all $t$, that is, that $Y+Z'=0$. 
%\mcc{TUESDAY: change ends here}
Hence $Y=-Z'\in I_x \cF_{S_x}$. Taking the image of $Y$ under $\cF(x)\to \g_x$ we obtain $0=[-Z']=[Y]=\sum_{i> k}\lambda^0_i [X_i]$, and since the $\{[X_i]\}_{i> k}$ for a basis of $\g_x$ we conclude $\lambda^0=0$.
\end{proof}

\section{Outlook: Holonomy transformations and the normal form of a singular foliation}\label{section:outlook}

In the present note we considered smoothness statements about the holonomy groupoid $H$. This, together with an analysis of the notion of holonomy for singular foliations, allows to construct a model for the singular foliation near a leaf, leading to the question of when the singular foliation is isomorphic to the local model. We provide a brief outline of this thoughts, which will be exposed thoroughly in a subsequent paper \cite{AZ2}.

 Let us start by recalling the following well-known facts about regular foliations:
\begin{itemize}
\item Geometrically, the holonomy of a co-dimension $q$ (regular) foliation $(M,F)$   at a point $x\in M$ is realized by a map $h \colon \pi_1(L_x) \to GermDiff(S)$, where $L_x$ is the leaf at $x$, $S$ is a transversal at $x$, and $GermDiff(S)$ is the space of germs of local diffeomorphisms of $S$. Its linearisation $Lin(h) \colon  \pi_1(L_x) \to GL(q)$ gives rise to a representation of the holonomy groupoid on the normal bundle $(TM/F)|_{L_x}$.

\item The (local) Reeb stability theorem (see \cite[Thm 2.9]{MM}) states that around a compact leaf $L$ with finite holonomy, the manifold is diffeomorphic to the quotient $\frac{\widetilde{L} \times \R^{q}}{\pi_1(L)}$, where $\widetilde{L}$ is the universal cover of the leaf, and $\pi_1(L)$ acts on $\widetilde{L} \times \R^{q}$ diagonally by deck transformations and linearized holonomy.
\end{itemize}
One encounters serious problems trying to define holonomy for singular foliations naively in terms of paths. Rather than using paths, we threat the notion of holonomy as follows. At the end we discuss the analog of the Reeb stability theorem.
\begin{itemize}
\item For $x, y$ in the same leaf consider transversals $S_x, S_y$ of the leaf at $x$ and $y$ respectively. There is a map $$\Phi_x^y : H(M,\cF)_x^y \to \frac{GermAut_{\cF}(S_x;S_y)}{exp(I_x\cF)|_{S_x}},$$ whose target is the space of germs of foliation-preserving local diffeomorphisms between $S_x$ and $S_y$, quotiented by the {group generated by} exponentials of elements in the maximal ideal $I_x\cF$.  The maps $\Phi_x^y$ assemble to a morphism of groupoids. We refer to elements of the target as ``\textit{holonomy transformations from $x$ to $y$}''.

\item The map $\Phi$ linearises to a morphism of groupoids $Lin(\Phi) \colon H(M,\cF) \to Iso(N)$, whose target is the groupoid of isomorphisms between the fibres of the (singular) normal bundle $N$ {to the leaves}  (linear holonomy). This map can be interpreted as the adjoint representation.

\item We can make sense of the quotient
\begin{align}\label{locm}
\frac{H_x\times N_xL}{H_x^x}.
\end{align} 
Namely, $H_x^x$ acts on  $N_xL$ by {$Lin(\Phi)$}. Our Thm. \ref{HL} ensures that 
$H_x$ and $H_x^x$ are both smooth, provided  the essential isotropy group of $L$ is closed, so that the quotient \eqref{locm} is a smooth manifold if $H_x^x$ is compact. Endowing $H_x$ with the foliation having just one leaf and $N_xL$ with the ($H_x^x$-invariant) singular foliation obtained linearizing $\cF$ at $x$, the above space  is further endowed with a singular foliation. 

The generalization of Reeb's stability theorem to singular foliations answers the question: When $L$ is compact and $H_x^x$ is a compact Lie group, under which other conditions (if any) is there a diffeomorphism of foliated manifolds between a neighborhood of $L$ and  the local model \eqref{locm}?

\end{itemize}

\end{document}